\documentclass[reqno,12pt]{amsart}
\usepackage[colorlinks=true, linkcolor=blue, citecolor=blue]{hyperref}

\usepackage{amssymb}
\usepackage{amsmath, graphicx, rotating}
\usepackage{color}
\usepackage{soul}
\usepackage[dvipsnames]{xcolor}

\allowdisplaybreaks
\usepackage{ifthen}
\usepackage{xkeyval}
\usepackage{todonotes}
\usepackage[T1]{fontenc}
\usepackage{lmodern}
\usepackage[english]{babel}

\usepackage{ upgreek }
\usepackage{stmaryrd}
\SetSymbolFont{stmry}{bold}{U}{stmry}{m}{n}
\usepackage{amsthm}
\usepackage{float}

\usepackage{ bbm }
\usepackage{ stmaryrd }
\usepackage{ mathrsfs }
\usepackage{ frcursive }
\usepackage{ comment }

\usepackage{pgf, tikz}
\usetikzlibrary{shapes}
\usepackage{varioref}
\usepackage{enumitem}

\definecolor{rouge}{rgb}{0.7,0.00,0.00}
\definecolor{vert}{rgb}{0.00,0.5,0.00}
\definecolor{bleu}{rgb}{0.00,0.00,0.8}
\usepackage[margin=1in]{geometry}
\newtheorem{theorem}{Theorem}[section]
\newtheorem*{theorem*}{Theorem}
\newtheorem{lemma}[theorem]{Lemma}

\newtheorem{proposition}[theorem]{Proposition}

\labelformat{hypothesis}{\textbf{M\kern-0.1mm#1}}

\newtheorem{condition}{Condition}

\newtheorem{conditionA}{A\kern-0.1mm}
\labelformat{conditionA}{\textbf{A\kern-0.1mm#1}}

\theoremstyle{definition}

\newtheorem{remark}[theorem]{Remark}

\def \eref#1{\hbox{(\ref{#1})}}

\numberwithin{equation}{section}

\def\geq{\geqslant}
\def\leq{\leqslant}

\def\RR{\mathbb{R}}
\def\PP{\mathbb{P}}
\def\EE{\mathbb{E}}

\def\vare{{\varepsilon}}
\def \eref#1{\hbox{(\ref{#1})}}

\def\EE{\mathbb{ E}}

\begin{document}

\title[Optimal convergence order for multi-scale SBE ]
{Optimal convergence order for multi-scale stochastic Burgers equation}

\author{Peng Gao}
\curraddr[Gao, P.]{ School of Mathematics and Statistics, and Center for Mathematics and Interdisciplinary Sciences, Northeast Normal University, Changchun 130024, China}
\email{gaopengjilindaxue@126.com}

\author{Xiaobin Sun}
\curraddr[Sun, X.]{ School of Mathematics and Statistics, and Research Institute of Mathematical Science, Jiangsu Normal University, Xuzhou, 221116, China}
\email{xbsun@jsnu.edu.cn}

\begin{abstract}
In this paper, we study the strong and weak convergence rates for multi-scale one-dimensional stochastic Burgers equation. Based on the techniques of Galerkin approximation, Kolmogorov equation and Poisson equation, we obtain the slow component strongly and weakly converges to the solution of the corresponding averaged equation with optimal orders 1/2 and 1 respectively. The highly nonlinear term in system brings us huge difficulties, we develop new technique to overcome these difficulties. To the best of our knowledge, this work seems to be the first result in which the optimal convergence orders in strong and weak sense for multi-scale stochastic partial differential equations with highly nonlinear term.
\end{abstract}

\date{\today}

\maketitle

\section{Introduction}

\vspace{0.1cm}
Many multi-scale problems arise from material sciences, chemistry, fluids dynamics, biology, ecology, climate dynamics and other application areas. Usually, this kind of systems usually consist of two components, which correspond to slow component $X^{\vare}$ and fast component $Y^{\vare}$ by using a small scale parament $\vare$. The issue how to describe asymptotic behavior of $(X^{\varepsilon},Y^{\varepsilon})$ as $\varepsilon\rightarrow0$ has attracted many people's attention. Averaging principles can help us study this kind of asymptotic behavior of $(X^{\varepsilon},Y^{\varepsilon})$, moreover, averaging principle provides a powerful tool for
simplifying nonlinear dynamical systems and obtain approximation solutions to differential equations arising from mechanics, mathematics, physics, control, and other areas. Furthermore, it can also help us understand and investigate the physical phenomenon described by nonlinear dynamical systems.

\vspace{0.1cm}
The theory of averaging principle of multi-scale system has a long history and rich results. The idea of averaging appeared in the perturbation theory developed by Clairaut, Laplace, and Lagrange in the 18th century. Then various averaging schemes (Gauss, Fatou, Delone-Hill) were widely applied in celestial mechanics in the 19th century. These were mainly formal techniques. Bogoliubov and Mitropolsky \cite{BM} first studied the averaging principle for the deterministic systems.

\vspace{0.1cm}
The averaging principle in the stochastic differential equations (SDEs for short) setup was first considered by Khasminskii in \cite{K1968}.
Generally speaking, a multi-scale stochastic system driven by Brownian motion has the following form
\begin{equation*}
\begin{array}{l}
(*)\left\{
\begin{array}{llll}
dX^{\varepsilon}_t=f(X^{\varepsilon}_t,Y^{\varepsilon}_t)dt+\sigma_{1}(X^{\varepsilon}_t, Y^{\varepsilon}_t)dW_t
\\ dY^{\varepsilon}_t=\frac{1}{\varepsilon}g(X^{\varepsilon}_t,Y^{\varepsilon}_t)dt+\frac{1}{\sqrt{\varepsilon}}\sigma_{2}(X^{\varepsilon}_t,Y^{\varepsilon}_t)dW_t.
\end{array}
\right.
\end{array}
\end{equation*}
\par
\it \textbf{Classical averaging principle of multi-scale stochastic system}:
\par
For system $(*)$, under some assumption of dissipation and ergodicity,
\textbf{\begin{eqnarray*}
\begin{array}{l}
\begin{array}{llll}
X^{\varepsilon}\rightarrow\bar{X},~~{\rm{as}}~\varepsilon\rightarrow0,
\end{array}
\end{array}
\end{eqnarray*}}
in weak sense, where $\bar{X}$ is the solution to the following averaged equation
\textbf{\begin{eqnarray*}
\begin{array}{l}
\begin{array}{llll}
d\bar{X}_t=\bar{f}(\bar{X}_t)dt+\bar{\sigma}_{1}(\bar{X}_t)dW_t
\end{array}
\end{array}
\end{eqnarray*}}
where $\bar{f}(x)=\int f(x,y)\mu_{x}(dy)$ and $\bar{\sigma}_{1}(x)=\left[\int \sigma_1(x,y)\sigma^{\ast}_1(x,y)\mu_x (dy)\right]^{1/2}$ are the averaged coefficients, with $\mu_{x}(dy)$ being the ergodic
invariant measure of the fast variables $Y_{x}$ with a frozen $x$ variable, namely,
$$dY_t=g(x,Y_t)dt+\sigma_{2}(x,Y_t)dW_t.$$
\rm
Then it is extended to stochastic partial differential equations (SPDEs for short) by Cerrai and Freidlin \cite{CF}. Since then, averaging principle for multi-scale SPDEs has become an active research area, we refer to \cite{BYY,B3,C1,FLL,PXY,WR,XML} and references therein for more interesting results on this topic.

\vspace{0.1cm}
It is well-known that a lot of models or equations in physics are accompanied by highly nonlinear term. Thus more and more multi-scale SPDEs with highly nonlinear term attracted much attention in this topic. For instance, stochastic one dimensional Burgers equations \cite{DSXZ}, stochastic Kuramoto-Sivashinsky equation \cite{GP1}, stochastic Schr\"{o}dinger equation \cite{GP2}, stochastic Klein-Gordon equation \cite{GP3}, stochastic Korteweg-de Vries equation \cite{GP4}, stochastic 2D Navier-Stokes equation \cite{GLSX}, stochastic 3D fractional Leray-$\alpha$ model \cite{LX2020}. Since the importance of the above models in physics, the rate of convergence has important practical significance, an important question arises from numerical simulation and scientific calculation:
\par
~~\par
\textbf{Can we obtain the optimal convergence order when $\varepsilon\rightarrow0$~?}
\par
~~\par
In this paper, we try to answer this question for multi-scale SPDEs with highly nonlinear term.
In the above references, the authors usually apply the technique of stopping time to deal with the nonlinear term, however it seems that this kind of technique will destroy the convergence rate, thus no satisfactory convergence rates were obtained in the mentioned references above.
In order to overcome these difficulties,  we develop \textit{new framework} to obtain the optimal convergence order. This kind of framework and method is successfully applied to multi-scales stochastic Burgers equations.

\vspace{0.1cm}
The Burgers equation
$$X_t-X_{\xi\xi}+XX_\xi=0$$
is used to study the turbulent fluid flow in \cite{B1974} and the references therein,
it arises as a 1-D simplification of the control system associated to the Navier-Stokes system, which has
been extensively studied these last years, see, for instance \cite{FI}. The Burgers equation has been
extensively used as a toy model to investigate properties of more complex systems in a rather simple
setting. This equation was introduced in the seminal paper \cite{Burgers} by Burgers. Both from a theoretical
and a numerical point of view, it already exhibits some key behaviors, such as interaction between the
nonlinearity and the smoothing effect. Since the stochastic Burgers equation can be used to model vortex lines in high-temperature superconductors, dislocations in disordered solids and kinetic roughening of interfaces in epitaxial growth, formation of
large-scale structures in the universe, constructive quantum field theory, etc., in recent years, many authors study the stochastic Burgers equation in \cite{B1994,DaD,Da1994,Da1995,DX,G1999,H2011,SWXY}.

\vspace{0.1cm}
In this paper, we consider the following multi-scale stochastic Burgers equation on the interval $[0,1]$:
\begin{equation}\left\{\begin{array}{l}\label{Equation}
\displaystyle
\frac{\partial X^{\varepsilon}_t(\xi)}{\partial t}=\big[\Delta X^{\varepsilon}_t(\xi)+\frac{1}{2}\frac{\partial}{\partial \xi}(X^{\varepsilon}_t(\xi))^2+F(X^{\varepsilon}_t, Y^{\varepsilon}_t)(\xi)\big]+\frac{\sqrt{Q_{1}}\partial W^{1}}{\partial t}(t,\xi),\quad X^{\varepsilon}_0=x \\
\frac{\partial Y^{\varepsilon}_t(\xi)}{\partial t}
=\frac{1}{\varepsilon}\big[\Delta Y^{\varepsilon}_t(\xi)+G(X^{\varepsilon}_t, Y^{\varepsilon}_t)(\xi)\big]+\frac{1}{\sqrt{\varepsilon}}\frac{\sqrt{Q_2}\partial W^{2}}{\partial t}(t,\xi),\quad Y^{\varepsilon}_0=y\\
X^{\varepsilon}_t(0)=X^{\varepsilon}_t(1)=Y^{\varepsilon}_t(0)
=Y^{\varepsilon}_t(1)=0, \end{array}\right.
\end{equation}
where $\varepsilon >0$ is a small parameter describing the ratio of time scales between the slow component $X^{\varepsilon}$
and fast component $Y^{\varepsilon}$. $\Delta$ is the Laplacian operator, measurable functions $F$ and $G$ satisfy some appropriate conditions. $Q_1, Q_2: H \rightarrow H$ are two non-negative selfadjoint trace operators. $Q_1$ is also a bounded linear operator. $\{W_t^1\}_{t\geq 0}$ and $\{W_t^2\}_{t\geq 0}$ are mutually independent cylindrical Wiener processes, which are defined on a complete filtered probability space $(\Omega,\mathscr{F},\{\mathscr{F}_{t}\}_{t\geq0},\mathbb{P})$.

\vspace{0.1cm}
In \cite{DSXZ}, Dong et. al. used the techniques of classical Khasminskii's time discretization and stopping time to prove the strong convergence, and applied the method of asymptotic expansion of solutions of Kolmogorov equations to study the weak convergence. Unfortunately, it does not achieve optimal convergence rate in the strong and weak sense. Recently, the technique of Poisson equation is widely used to study the optimal convergence rate in the multi-scale SDEs or SPDEs, we refer to \cite{B3,RSX2021,RX2021,SXX2022}. More applications of Poisson equation, see e. g. \cite{PV2001,PV2003} and references therein.

\vspace{0.1cm}
Thus, this paper aims to improve the strong and weak convergence orders obtained  in \cite{DSXZ} to optimal orders. More precisely, using the techniques of Galerkin approximation, Kolmogorov equation and Poisson equation, we obtain the slow component strongly and weakly converges to the solution of the corresponding averaged equation with optimal orders 1/2 and 1 respectively. It is worth to point that in order to overcome the difficulty cased by the high nonlinear term, some accurate estimates, such as the exponential order, of the solution are needed.
To the best of our knowledge, it seems that this paper is the first attempt in studying the optimal convergence rates for multi-scale SPDEs with highly nonlinear term.
Moreover, since the optimal convergence order is achieved, then the subsequent problems, such as the central limit type theorem or diffusion approximation, for multi-scale SPDEs with highly nonlinear term can be considered, which are left for our further works.

\vspace{0.1cm}
The rest of the paper is organized as follows. In Section 2, we give some notation and under some suitable assumptions, we formulate our main results. Section 3 are devoted to some a priori estimates of the finite dimensional frozen equation and study the corresponding Poisson equation. Section 4 and Section 5 are devoted to proving the strong convergence and weak convergence, respectively. In the Appendix 6, we recall some useful inequalities and give the detailed proofs of the Galerkin approximation. Throughout this paper, $C$, $C_{T}$ and $C_{p,T}$ stand for constants whose value may change from line to line, and $C_{T}$ and $C_{p,T}$ is used to emphasize  that the constant depend on $T$ and $p,T$ respectively.

\section{Main results} \label{Sec Main Result}

\subsection{Notations}
For any $p\geq 1$, let $L^p(0,1)$ be the space of $p$-th integrable real-valued functions on the interval $[0,1]$. Denote $H=\{f\in L^2(0,1): f(0)=f(1)=0\}$ with the usual norm $|\cdot|$ and the inner product  $\langle\cdot,\cdot\rangle$, respectively. $\mathscr{B}(H)$ is the collection of all measurable function $\varphi(x): H\rightarrow \RR$.

For any $k\in \mathbb{N}_{+}$, define
$$C^k(H):=\{\varphi\in \mathscr{B}(H): \varphi \mbox{ and all its Fr\'{e}chet derivative up to order } k \mbox{ are continuous}\},$$
$$C^k_b(H):=\{\varphi\in C^k(H): \mbox{ for } 1\le i\le k, \mbox{ the}~i\mbox{-th}~\mbox{Fr\'{e}chet derivative of } \varphi \mbox{ are bounded}\}.$$

For any $\varphi\in C^3(H)$, by Riesz representation theorem, we often identify the first Fr\'{e}chet derivative $D\varphi(x)\in \mathcal{L}(H, \RR)\cong H$, the second derivative $D^2\varphi(x)$ as a linear operator in $\mathcal{L}(H,H)$ and the third derivative $D^3\varphi(x)$ as a linear operator in $\mathcal{L}(H,\mathcal{L}(H,H))$, i.e.,
\begin{eqnarray*}
&&D\varphi(x) \cdot h = \langle D\varphi(x), h \rangle,\quad h\in H,\\
&&D^2\varphi(x) \cdot (h,k) = \langle D^2\varphi(x)\cdot h, k \rangle,  \quad  h,k \in H,\\
&&D^3\varphi(x) \cdot (h,k,l)= [D^3\varphi(x)\cdot h]\cdot (k,l),\quad h,k,l\in H.
\end{eqnarray*}

For $k\in\mathbb{N}_{+}$, $W^{k, p}(0,1)$ is the Sobolev space
of all functions in $L^p(0,1)$ whose differentials belong to $L^p(0,1)$ up to the order $k$. The usual Sobolev space $W^{k, p}(0,1)$ can be extended to
the $W^{s, p}(0,1)$, for $s\in\mathbb{R}$. Set $H^k(0,1):=W^{k, 2}(0,1)$
and denote by $H^1_0(0,1)$ the subspace of $H^1(0,1)$ of all functions
whose trace at $0$ and $1$ vanishes.

Denote Laplacian operator $\Delta$ by
\begin{align*}
Ax := \Delta x := \frac{\partial^2}{\partial \xi^2}x,
\quad  x\in \mathscr{D}(A)=H^2(0,1) \cap H^1_0(0,1),
\end{align*}
where $H^1_0(0,1)$ the subspace of $H^1(0,1)$ of all functions
whose trace at $0$ and $1$ vanishes. It is well known that $\Delta$ is the infinitesimal generator of a strongly continuous semigroup
$\{e^{t \Delta}\}_{t\geq0}$. The eigenfunctions of $\Delta$ are given by
$e_k(\xi)=\sqrt{2}\sin(k\pi\xi)$, $\xi\in[0,1],k\in \mathbb{N}_{+}$,
with the corresponding eigenvalues $-\lambda_k$ with $\lambda_k = k^2\pi^2$.

For any $s\in\RR$, let $(-A)^s$ be the power of the operator $-A$, i.e.,
 $$H^s:=\mathscr{D}((-A)^{s/2}):=\left\{u=\sum_{k\in\mathbb{N}_{+}}u_ke_k: u_k\in \mathbb{R},~\sum_{k\in\mathbb{N}_{+}}\lambda_k^{s}u_k^2<\infty\right\}$$
and
 $$(-A)^{s/2}u:=\sum_{k\in \mathbb{N}}\lambda_k^{s/2} u_ke_k,~~u\in\mathscr{D}((-A)^{s/2}),$$
with the associated norm $\|u\|_{s}:=|(-A)^{s/2}u|=\left(\sum_{k\in\mathbb{N}_{+}}\lambda_k^{s} u^2_k\right)^{1/2}$. It is easy to see $\|\cdot\|_0=|\cdot|$ and the following inequalities hold:
\begin{eqnarray}
&&\|e^{tA}x\|_{\sigma_2}\leq C_{\sigma_1,\sigma_2}t^{-\frac{\sigma_2-\sigma_1}{2}}e^{-\frac{\lambda_1 t}{2}}\|x\|_{\sigma_1},\quad x\in H^{\sigma_2},\sigma_1\leq\sigma_2, t>0;\label{P3}\\
&&|e^{tA}x-x |\leq Ct^{\frac{\sigma}{2}}\|x\|_{\sigma},\quad x\in H^{\sigma}, \sigma\geq 0,t\geq 0.\label{P4}
\end{eqnarray}

Define the bilinear operator $B(x,y): L^2(0, 1) \times H^1_0(0, 1)\rightarrow H^{-1}_0(0, 1)$ by
$$ B(x,y)=x\cdot\partial_{\xi} y,$$
and the trilinear operator $ b(x,y,z): L^2(0, 1) \times H^1_0(0, 1)\times L^2(0, 1) \rightarrow \RR$ by
$$ b(x,y,z)
=\int_0^1x(\xi) \partial_\xi y(\xi)z(\xi) d\xi.$$
For convenience, for $x\in H^1_0(0, 1)$, set $B(x)=B(x,x)$ (Some properties of $b$ and $B$ can be founded in the Appendix).

\vskip 0.3cm

With the above notations, the system \eref{Equation} can be rewritten as: 
\begin{equation}\left\{\begin{array}{l}\label{main equation}
\displaystyle
dX^{\vare}_t=[AX^{\vare}_t+B(X^{\vare}_t)+F(X^{\vare}_t, Y^{\vare}_t)]dt+\sqrt{Q_1}dW^{1}_t,\quad X^{\vare}_0=x;\\
dY^{\vare}_t=\frac{1}{\vare}[AY^{\vare}_t+G(X^{\vare}_t, Y^{\vare}_t)]dt+\frac{1}{\sqrt{\vare}}\sqrt{Q_2}dW^{2}_t,\quad Y^{\vare}_0=y;\\
X^{\vare}_t(0)=X^{\vare}_t(1)=Y^{\vare}_t(0)=Y^{\vare}_t(1)=0,\end{array}\right.
\end{equation}
where $W^{1}_t$ and $W^{2}_t$ are two independent  cylindrical Wiener processes given by
\begin{align} \label{Q1WienerPro}
W^{i}_t=\sum^\infty_{k=1}\beta^{i,k}_te_k, \quad t\geq0,\quad i=1,2,
\end{align}
here $\{\beta^{i,k}\}_{k\in \mathbb{N}_{+}}$ are mutually independent one-dimensional standard Brownian motions, $i=1,2$. Assume that $W^{1}_t$ and $W^{2}_t$ are independent, $Q_i e_k=\alpha_{i,k}e_k$, $i=1,2$, $k\in\mathbb{N}_{+}$.

\subsection{Assumptions}
Now, we assume the following conditions  on the coefficients $F, G: H\times H \rightarrow H$ throughout the paper:
\begin{conditionA}\label{A1}
There exist positive constants $C$, $L_F$ and $L_{G}$ satisfying $\lambda_1-2L_F>0$ and $\lambda_1-L_G>0$ such that for any $x_1,x_2,y_1,y_2\in H$,
\begin{eqnarray}
&&\left|F(x_1,y_1)-F(x_2,y_2)\right|\leq L_F|x_1-x_2|+C|y_1-y_2|,\label{LF}\\
&&\left|G(x_1, y_1)-G(x_2, y_2)\right|  \leq C|x_1-x_2| + L_{G}|y_1-y_2|,\label{LG}\\
&&\sup_{y\in H}|F(x,y)|\leq L_F|x|+C,\quad \sup_{x\in H}|G(x,y)|\leq L_G|y|+C. \label{F2.21}
\end{eqnarray}
\end{conditionA}

\begin{conditionA}\label{A2}
Assume that there exists $\tau\in (0,2)$ such that the following directional derivatives are well-defined and satisfy:
\begin{equation}\left\{\begin{array}{l}\label{RK}
\displaystyle
|D_{x}K(x,y)\cdot h|\leq C |h| \quad \text{and} \quad |D_{y}K(x,y)\cdot h|\leq C |h|,\quad \forall x,y, h\in H ;\\
|D_{xx}K(x,y)\cdot(h,k)|\leq C |h|\|k\|_{\tau},\quad \forall x,y, h\in H, k\in H^{\tau} ; \\
|D_{yy}K(x,y)\cdot(h,k)|\leq C |h|\|k\|_{\tau},\quad \forall x,y, h\in H, k\in H^{\tau} ;   \\
|D_{xy}K(x,y)\cdot (h,k)|\leq C |h|\|k\|_{\tau}, \quad \forall x,y, h\in H, k\in H^{\tau} ; \\
|D_{xy}K(x,y)\cdot (h,k)|\leq C \|h\|_{\tau}|k|, \quad \forall x,y, k\in H, h\in H^{\tau} ;\\
|D_{xyy}K(x,y)\cdot(h,k,l)|\leq C|h|\|k\|_{\tau}\|l\|_{\tau}, \quad \forall x,y, h\in H, k,l\in H^{\tau} ;  \\
|D_{yyy}K(x,y)\cdot(h,k,l)|\leq C|h|\|k\|_{\tau}\|l\|_{\tau}, \quad \forall x,y, h\in H, k,l\in H^{\tau} ; \\
|D_{xxy}K(x,y)\cdot(h,k,l)|\leq C|h|\|k\|_{\tau}\|l\|_{\tau}, \quad \forall x,y, h\in H, k,l\in H^{\tau} ; \\
|D_{xxx}K(x,y)\cdot(h,k,l)|\leq C|h|\|k\|_{\tau}\|l\|_{\tau}, \quad \forall x,y, h\in H, k,l\in H^{\tau},
\end{array}\right.
\end{equation}
where  $K$ is taken by $F$ and $G$, $D_{x}K(x,y)\cdot h$ is the directional derivative of $K(x,y)$ in the direction $h$ with respective to $x$, other notations can be interpreted similarly.
\end{conditionA}

\begin{conditionA}\label{A3}
There exist two constants $\alpha,\beta\in (0,1)$ such that $F: H^{\beta}\times H^{\beta}\rightarrow H^{\alpha}$ and for any $x,y,y_1,y_2\in H^{\beta}$,
\begin{eqnarray}
&&\|F(x, y_1)-F(x, y_2)\|_{\alpha}\leq C(1+\|x\|_{\beta}+\|y_1\|_{\beta}+\|y_2\|_{\beta})\|y_1-y_2\|_{\beta},\label{FH}\\
&&\|F(x,y)\|_{\alpha}\leq C(1+\|x\|_{\beta}+\|y\|_{\beta}).\label{BF}
\end{eqnarray}
\end{conditionA}

\begin{conditionA}\label{A4} Suppose that
\begin{eqnarray}
\Sigma_{k}\lambda_k\alpha_{1,k}<\infty\quad \text{and}\quad\Sigma_{k}\lambda^{\beta-1}_k\alpha_{2,k}<\infty,\label{Noise}
\end{eqnarray}
where $\beta$ is the constant in assumption \ref{A3}.
\end{conditionA}

\begin{remark} Under the conditions \eref{LF} and \eref{LG}, for any given $\varepsilon>0$ and initial value $(x,y)\in H\times H$, it is easy to prove the system \eref{main equation} exists a unique mild solution $(X^{\vare}_t, Y^{\vare}_t)\in H\times H$, i.e., $\PP$-a.s.,
\begin{equation}\left\{\begin{array}{l}\label{A mild solution}
\displaystyle
X^{\varepsilon}_t=e^{tA}x+\int^t_0e^{(t-s)A}B(X^{\varepsilon}_s)ds+\int^t_0e^{(t-s)A}F(X^{\varepsilon}_s, Y^{\varepsilon}_s)ds+\int^t_0 e^{(t-s)A}\sqrt{Q_1}dW_s^1,\\
Y^{\varepsilon}_t=e^{tA/\varepsilon}y+\frac{1}{\varepsilon}\int^t_0e^{(t-s)A/\varepsilon}G(X^{\varepsilon}_s,Y^{\varepsilon}_s)ds
+\frac{1}{\sqrt{\varepsilon}}\int^t_0 e^{(t-s)A/\varepsilon}\sqrt{Q_2}dW_s^2.
\end{array}\right.
\end{equation}
\end{remark}
\begin{remark}
The condition $\lambda_{1}-L_{G}>0$ in \ref{A1} is called the strong dissipative condition, which is used to prove the existence and uniqueness of the invariant measures and the exponential ergodicity of the transition semigroup of the frozen equation. The condition $\lambda_{1}-2L_{F}>0$ and \eref{F2.21} in  \ref{A1} is used to ensure the solution $X^{\varepsilon}_t$ and $\bar{X}_t$ have finite exponential moment. 
The condition \ref{A2}  is used to study the regularity of the solution of the corresponding Poisson equation. The condition  \ref{A3} is used to obtain the optimal strong and weak convergence order. The condition \ref{A4} is used to study the regularity of the stochastic convolutions $\int^t_0 e^{(t-s)A}\sqrt{Q_1}dW_s^1$ and $\int^t_0 e^{(t-s)A}\sqrt{Q_2}dW_s^2$.
\end{remark}

\begin{remark} \label{EA2}
Here we give an example that the condition \ref{A2} holds.
Let $K$ be the Nemytskii operator associated with a function $f:\RR\times \RR\rightarrow \RR$, i.e., $K(x,y)(\xi):=f(x(\xi),y(\xi))$. Then the following directional derivatives are well-defined and belong to $H$,
\begin{equation}\left\{\begin{array}{l}\label{RKE}
\displaystyle
D_{x}K(x,y)\cdot h=\partial_x f(x,y)h \quad \text{and} \quad D_{y}K(x,y)\cdot h=\partial_y f(x,y)h ,\quad \forall x,y, h\in H ;   \\
D_{xx}K(x,y)\cdot(h,k)=\partial_{xx}f(x,y)hk,\quad \forall x,y, h, k\in L^{\infty}(0,1) ;   \\
D_{yy}K(x,y)\cdot(h,k)=\partial_{yy}f(x,y)hk,\quad \forall x,y, h\in H, k\in L^{\infty}(0,1) ;  \\
D_{xy}K(x,y)\cdot (h,k)=\partial_{xy}f(x,y)hk, \quad \forall x,y, h\in H, k\in L^{\infty}(0,1) ;  \\
D_{xy}K(x,y)\cdot (h,k)=\partial_{xy}f(x,y)hk, \quad \forall x,y, k\in H, h\in L^{\infty}(0,1) ;\\
D_{xyy}K(x,y)\cdot(h,k,l)=\partial_{xyy}f(x,y)hkl, \quad \forall x,y, h\in H, k,l\in L^{\infty}(0,1) ; \\
D_{yyy}K(x,y)\cdot(h,k,l)=\partial_{yyy}f(x,y)hkl, \quad \forall x,y, h\in H, k,l\in L^{\infty}(0,1) ; \\
D_{xxy}K(x,y)\cdot(h,k,l)=\partial_{xxy}f(x,y)hkl, \quad \forall x,y, h\in H, k,l\in L^{\infty}(0,1);\\
D_{xxx}K(x,y)\cdot(h,k,l)=\partial_{xxx}f(x,y)hkl, \quad \forall x,y, h\in H, k,l\in L^{\infty}(0,1),
\end{array}\right.
\end{equation}
where we assume that all the partial derivatives of $f$ appear above are all uniformly bounded. Note that $H^{\tau}\subset L^{\infty}(0,1)$ for any $\tau>1/2$, thus it is easy to see that $K$ satisfies assumption \ref{A2} with any $\tau\in (1/2,2)$.
\end{remark}

\begin{remark}
Here we give an example that the condition \ref{A3} holds. Recall the notation in Remark \ref{EA2}, if $\partial_y f(\cdot,\cdot):\RR\times \RR\rightarrow \RR$ is a global Lipschitz function, we refer to \cite[Proposition 2.1]{B3} or \cite[Section 3.2]{BD}, it follows that for any $\alpha\in (0,1)$, $\tau\in (0, 1-\alpha)$,
\begin{eqnarray*}
\|(-A)^{\alpha/2}\left[f(x,y_1)-f(x,y_2)\right]\|_{L^2}\leq\!\!\!\!\!\!\!\!&&C\Big(1+\|(-A)^{\frac{\alpha+\tau}{2}}x\|_{L^4}+\|(-A)^{\frac{\alpha+\tau}{2}}y_1\|_{L^4}\nonumber\\
&&\quad+\|(-A)^{\frac{\alpha+\tau}{2}}y_2\|_{L^4}\Big)\|(-A)^{\frac{\alpha+\tau}{2}}(x-y)\|_{L^4},
\end{eqnarray*}
where $\|g\|_{L^p}:=\left[\int^1_0 |g(\xi)|^p d\xi \right]^{1/p}$. Then by Sobolev inequality $\|x\|_{L^4}\leq C\|x\|_{\frac{1}{4}}$, we have
\begin{eqnarray*}
\|F(x,y_1)-F(x,y_2)\|_{\alpha}\leq\!\!\!\!\!\!\!\!&& C\Big(1+\|x\|_{\alpha+\tau+1/4}\nonumber\\
&&\quad+\|y_1\|_{\alpha+\tau+1/4}+\|y_2\|_{\alpha+\tau+\frac{1}{4}}\Big)\|x-y\|_{\alpha+\tau+\frac{1}{4}}.
\end{eqnarray*}
Furthermore, if $f(\cdot,\cdot):\RR\times \RR\rightarrow \RR$ is a global Lipschitz function, then it follows
\begin{eqnarray*}
\|F(x,y)\|_{\alpha}\leq\!\!\!\!\!\!\!\!&&C\Big(1+\|(-A)^{\frac{\alpha+\tau}{2}}x\|_{L^2}+\|(-A)^{\frac{\alpha+\tau}{2}}y\|_{L^2}\Big)\nonumber\\
\leq\!\!\!\!\!\!\!\!&&C\Big(1+\|x\|_{\alpha+\tau}+\|y\|_{\alpha+\tau}\Big).
\end{eqnarray*}
Consequently, condition \ref{A3} holds for any $\alpha\in (0,3/4)$ and $\beta=\alpha+\tau+1/4$ with $\tau<3/4-\alpha$.
\end{remark}

\subsection{Main results} Let $\bar{X}$ be the solution of the corresponding averaged equation:
\begin{equation}\left\{\begin{array}{l}
\displaystyle d\bar{X}_{t}=\left[A\bar{X}_{t}+B(\bar{X}_{t})+\bar{F}(\bar{X}_{t})\right]dt+\sqrt{Q_1}d W_{t}^1,\\
\bar{X}_{0}=x\in H,\end{array}\right. \label{1.3}
\end{equation}
where the averaged coefficient $\bar{F}(x):=\int_{H}F(x,y)\mu^{x}(dy)$ with $\mu^{x}$ is the unique invariant measure of the transition semigroup of the frozen equation
\begin{eqnarray}\label{FZE}
\left\{ \begin{aligned}
&dY_{t}=\left[AY_{t}+G(x,Y_{t})\right]dt+\sqrt{Q_2}d W_{t}^2,\\
&Y_{0}=y\in H.
\end{aligned} \right.
\end{eqnarray}

\vspace{0.1cm}
Now, we state our main result here.
\begin{theorem} (\textbf{Strong convergence})\label{main result1}
Suppose  that assumptions \ref{A1}-\ref{A4} hold. For any initial value $(x,y)\in H^{\theta}\times H$ with $\theta\in (0,1]$, $p\geq 2, T>0$ and small $\vare>0$, there exists a constant $C>0$ depending on $p,T,|x|,\|x\|_{\theta},|y|$ such that
\begin{align}
\mathbb{E} \left(\sup_{t\in [0,T]} |X_{t}^{\vare}-\bar{X}_{t}|^{p} \right)\leq C\vare^{p/2}. \label{ST2}
\end{align}
\end{theorem}
\begin{remark}
Instead of using the technique of stopping time in \cite{DSXZ}, through by some accurate estimations of the solution and using the skill of Poisson equation,  we obtain the strong convergence order $1/2$, which is the optimal order in the strong sense (see \cite[Example 1]{L2010}).
\end{remark}
\begin{theorem} (\textbf{Weak convergence})\label{main result 2}
Suppose that the assumptions \ref{A1}-\ref{A4} hold. Then for any test function $\phi\in C^3_b(H)$, initial value $(x,y)\in H^{\theta}\times H$ with $\theta\in (0,1]$, $T>0$ and $\vare>0$, there exists a constant $C>0$ depending on $T,|x|,\|x\|_{\theta},|y|$ such that
\begin{align}
\sup_{t\in [0,T]}\left|\mathbb{E}\phi(X_{t}^{\vare})-\EE\phi(\bar{X}_{t})\right|\leq C\vare. \label{WC}
\end{align}
\end{theorem}
\begin{remark}
Compare with the model considered in \cite{DSXZ}, we here allow the additive noise in the slow component when studying the weak convergence. Meanwhile, using the techniques of Kolmogorov equation and Poisson equation, we obtain the optimal weak convergence order 1, which improves the order $1-r$ for any $r\in(0,1)$ obtained in  \cite{DSXZ}.
\end{remark}

\section{The a priori estimates}
Note that the operator $A$ is  an unbounded operator, we use Galerkin approximation to reduce the infinite dimensional problem to a finite dimension, i.e., considering the following approximation:
\begin{equation}\left\{\begin{array}{l}\label{Ga mainE}
\displaystyle
dX^{m,\vare}_t=[AX^{m,\vare}_t+B^m(X^{m,\vare}_t)+F^m(X^{m,\vare}_t, Y^{m,\vare}_t)]dt+\sqrt{Q_1}d\bar W^{1,m}_t,\  X^{m,\vare}_0=x^{m}\in H_m,\\
dY^{m,\vare}_t=\frac{1}{\vare}[AY^{m,\vare}_t+G^m(X^{m,\vare}_t, Y^{m,\vare}_t)]dt+\frac{1}{\sqrt{\varepsilon}}\sqrt{Q_2}d\bar W^{2,m}_t,\quad Y^{m,\vare}_0=y^m\in H_m,\end{array}\right.
\end{equation}
where $m\in \mathbb{N}_{+}$, $H_{m}:=\text{span}\{e_{k};1\leq k \leq m\}$, $\pi_{m}$ is the orthogonal projection of $H$ onto $H_{m}$, $x^{m}:=\pi_m x, y^m:=\pi_m y$ and
\begin{eqnarray*}
&& B^m(x):=\pi_m B(x^{m}),\quad F^{m}(x,y):=\pi_m F(x^{m},y^{m}),\quad G^{m}(x,y):=\pi_m G(x^{m}, y^{m}),\\
&&\bar W^{1,m}_t:=\sum^m_{k=1}W^{1,k}_{t}e_k,\quad \bar W^{2,m}_t:=\sum^m_{k=1}W^{2,k}_{t}e_k.
\end{eqnarray*}
Similarly, we consider the following approximation to the averaged equation \eref{1.3}:
\begin{equation}\left\{\begin{array}{l}
\displaystyle d\bar{X}^m_{t}=\left[A\bar{X}^m_{t}+B^m(\bar{X}^m_{t})+\bar{F}^m(\bar{X}^m_{t})\right]dt+\sqrt{Q_1}d \bar{W}^{1,m}_{t},\\
\bar{X}^m_{0}=x^m,\end{array}\right. \label{Ga 1.3}
\end{equation}
where $\bar{F}^m(x):=\int_{H_m}F^m(x,y)\mu^{x,m}(dy)$, and $\mu^{x,m}$ is the unique invariant measure of the transition semigroup of the following frozen equation:
\begin{equation}\left\{\begin{array}{l}
\displaystyle dY^{x,y,m}_t=[AY^{x,y,m}_t+G^m(x,Y^{x,y,m}_t)]dt+\sqrt{Q_2}d\bar{W}^{2,m}_t,\\
Y^{x,y,m}_0=y\in H_m.\end{array}\right. \label{Ga FZE}
\end{equation}

Next, we give some a priori estimates of the solution $\bar{X}^m_t$ of the averaged equation \eref{Ga 1.3} and the solution $(X_{t}^{m,\varepsilon}, Y_{t}^{m,\varepsilon})$ of system \eref{Ga mainE} in subsections 3.1 and 3.2 respectively.

\subsection{The a priori estimates of the solution $\bar{X}^{m}_t$}
Recall the finite dimensional averaged equation \eref{Ga 1.3}.
Note that
$$\bar{F}^m(x):=\int_{H_m}F^m(x,z)\mu^{x,m}(dz)=\lim_{t\rightarrow \infty}\EE F^m(x, Y^{x,y,m}_t),\quad \forall y\in H_m.$$
By assumptions \ref{A1} and \ref{A2} , through a straightforward computation, it is easy to check that
\begin{eqnarray}
&&|D\bar{F}^m(x)\cdot h|\leq C |h| \quad \forall x, h\in H_m ,\label{D1Fm}\\
&&|D^2\bar{F}^m(x)\cdot(h,k)|\leq C |h|\|k\|_{\tau},\quad \forall x, h, k\in H_m,\label{D2Fm}\\
&&|D^3\bar{F}^m(x)\cdot(h,k,l)|\leq C|h|\|k\|_{\tau}\|l\|_{\tau}, \quad \forall x,h,k,l\in H_m,\label{D3Fm}
\end{eqnarray}
where $\tau$ is the constant in assumption \ref{A2}.

As a consequence, \eref{Ga 1.3} admits a unique mild solution $\bar{X}^m_{t}$, i.e.,
\begin{eqnarray*}
\bar X^m_t=e^{tA}x^m+\int^t_0e^{(t-s)A} B^m(\bar X^m_s)ds+\int^t_0e^{(t-s)A} \bar F^m(\bar X^m_s)ds+\int^t_0 e^{(t-s)A}\sqrt{Q_1}\bar{W}^{1,m}_s.
\end{eqnarray*}
The following is the a prior estimate of the solution $\bar{X}^m_{t}$:
\begin{lemma} \label{BarXgamma} (1) For any $p\geq 1$ and $T>0$, there exists a constant $C_{p,T}>0$ such that for any $x\in H$, we have
\begin{align}
\sup_{m\in\mathbb{N}_{+} }\EE\left(\sup_{t\in [0,T]}|\bar{X}^m_t|^p\right)\leq C_{p,T}(1+|x|^p).\label{A.1.0}
\end{align}
(2) For any $T>0$, there exists a constant $C_T>0$ such that for any $x\in H$  and $\delta<\frac{\lambda_1-2L_{F}}{4\|Q_1\|}$, we have
\begin{align}
\sup_{m\in\mathbb{N}_{+} }\EE\left(e^{\delta\sup_{t\leq T}|\bar{X}^m_t|^2+\delta\int^T_0\|\bar{X}^m_s\|^2_1ds}\right)\leq C_T(1+|x|)e^{\delta |x|^2}.\label{A.1.1}
\end{align}

(3) For any $T>0$, $p\geq 1$, $\gamma\in(1,3/2)$ and $x\in H^{\theta}$ with $\theta\in [0, \gamma]$,
there exist constants $d>1$ and $C_{T}>0$ such that for any $t\in (0, T]$, we have
\begin{align}
\sup_{m\in\mathbb{N}_{+} }\left[\EE\|\bar{X}^m_{t}\|^p_{\gamma}\right]^{1/p}\leq C_{T} t^{-\frac{(\gamma-\theta) }{2}}(1+|x|^d+\|x\|_{\theta}). \label{A.1.2}
\end{align}
\end{lemma}

\begin{proof}
(1) Note that $\text{Tr}Q_1<\infty$, by a minor revision as in the proof of \cite[Lemma 3.1]{DSXZ}, we can easily obtain \eref{A.1.0}.

(2) The proof of \eref{A.1.1} is inspired from \cite[Proposition 5.11]{Da}. However, for the reason of technique which will be used later. Here we need a stronger result. Set
$$
Z^m_t=|\bar{X}^m_t|^2+\int^t_0 \|\bar{X}^m_s\|^2_1ds,\quad t\in [0,T].
$$
Then we have
$$
dZ^m_t=\left[-\|\bar{X}^m_t\|^2_1+\text{Tr}Q^m_1+2\langle \bar{X}^m_t, \bar{F}(\bar{X}^m_t)\rangle\right]dt+2\langle \bar{X}^m_t, \sqrt{Q_1}dW^1_t\rangle,
$$
where $Q^m_1=\pi_m Q^1$.

By It\^o's formula, it follows that for any $\delta>0$,
\begin{eqnarray}
e^{\delta Z^m_t}= \!\!\!\!\!\!\!\!&&e^{\delta |x|^2}+\int^t_0\delta e^{\delta Z^m_s} \left[-\|\bar{X}^m_s\|^2_1+2\delta|\sqrt{Q_1}\bar{X}^m_s|^2+\text{Tr}Q^m_1+2\langle \bar{X}^m_s, \bar{F}^m(\bar{X}^m_s)\rangle\right]ds\nonumber\\
&&+\int^t_0 2\delta e^{\delta Z^m_s}\langle \bar{X}^m_s, \sqrt{Q_1}d\bar{W}^{1,m}_s\rangle.\label{F6.5}
\end{eqnarray}
Note that by \eref{F2.21}, it follows
$$
|\bar{F}^m(x)|\leq\int_{H}\left|F^m(x,y)\right|\mu^{x,m}(dy)\leq \int_{H}(L_F|x|+C)\mu^{x,m}(dy)\leq L_F|x|+C.
$$
Then for any $\delta\in (0, \frac{\lambda_1-2 L_{F}}{2\|Q_1\|})$, where $\|Q_1\|$ is the usual operator norm of $Q_1$, we get
\begin{eqnarray*}
&&-\|\bar{X}^m_s\|^2_1+2\delta|\sqrt{Q_1}\bar{X}^m_s|^2+2\langle \bar{X}^m_s, \bar{F}^m(\bar{X}^m_s)\rangle\\
\leq\!\!\!\!\!\!\!\!&& -(\lambda_1-2L_{F})|\bar{X}^m_s|^2+2\delta\|Q_1\||\bar{X}^m_s|^2+C|\bar{X}^m_s|\leq C,
\end{eqnarray*}
which combines with \eref{F6.5}, we obtain
\begin{eqnarray*}
\EE e^{\delta Z^m_t}\leq e^{\delta |x|^2}+\delta (\text{Tr}Q_1+C)\int^t_0\EE e^{\delta Z^m_s}ds.
\end{eqnarray*}
By Gronwall's lemma, we obtain that
\begin{eqnarray}
\EE e^{\delta Z^m_t}\leq e^{\delta |x|^2+t\delta \left( \text{Tr}Q_1+C\right)}.\label{F6.6}
\end{eqnarray}

Using \eref{F6.5}, \eref{F6.6} and Burkholder-Davis-Gundy's inequality, for any $\delta\in (0, \frac{\lambda_1-2 L_{F}}{4\|Q_1\|}) $,
\begin{eqnarray*}
\EE\left(\sup_{t\leq T}e^{\delta Z^m_t}\right)\leq\!\!\!\!\!\!\!\!&&e^{\delta |x|^2}+\delta (\text{Tr}Q_1+C)\int^T_0\EE\left( e^{\delta Z^m_s} \right)ds+C\delta\text{Tr}Q_1\EE\left[\int^T_0e^{2\delta Z^m_s}|\bar{X}^m_s|^2ds\right]^{1/2} \\
\leq\!\!\!\!\!\!\!\!&&e^{\delta |x|^2}+\delta T(\text{Tr}Q_1+C)e^{\delta |x|^2+T\delta \left( \text{Tr}Q_1+C\right)}+C_T\delta\text{Tr}Q_1(1+|x|)e^{\delta |x|^2+T\delta \left( \text{Tr}Q_1+C\right)}\\
\leq\!\!\!\!\!\!\!\!&&C_T(1+|x|)e^{\delta |x|^2},
\end{eqnarray*}
which implies that \eref{A.1.1} holds.

(3) Recall that
\begin{align*}
\bar{X}^m_t=e^{tA}x^m+\int^t_0e^{(t-s)A}B^m(\bar{X}^m_s)ds+\int^t_0e^{(t-s)A}\bar{F}^m(\bar{X}^m_s)ds+\int^t_0 e^{(t-s)A}\sqrt{Q_1}d\bar{W}^{1,m}_s.
\end{align*}


According to $\eref{P3}$ and Lemma \ref{Property B1}, we have  for any $\gamma\in(1,3/2)$,
\begin{eqnarray*}
\Big\|\int^t_0e^{(t-s)A}B^m(\bar{X}^m_s)ds\Big\|_{\gamma}\leq \!\!\!\!\!\!\!\!&& \int^t_0\big\|e^{(t-s)A}B^m(\bar{X}^m_s)\big\|_{\gamma}ds \nonumber\\
\leq \!\!\!\!\!\!\!\!&& C\int^t_0 (t-s)^{\frac{-\alpha_{3}-\gamma}{2}}\|B^m(\bar{X}^m_s)\|_{-\alpha_{3}}ds  \nonumber\\
\leq \!\!\!\!\!\!\!\!&&
C\int^t_0(t-s)^{\frac{-\alpha_{3}-\gamma}{2}}\|\bar{X}^m_s\|_{\alpha_{1}}\|\bar{X}^m_s\|_{\alpha_{2}+1}ds,
\end{eqnarray*}
where we choose positives constants $\alpha_{1}$ and $\alpha_{2}$  be small enough such that
$1+\alpha_{1}+\alpha_{2} \in (1,\gamma)$, then find a proper $\alpha_{3}\in [1/2, 2-\gamma)$ satisfying $ \alpha_{1}+\alpha_{2}+\alpha_{3}>\frac{1}{2} $.
(For instance, choosing by $\alpha_{3}=\frac{1}{2}$ and $\alpha_{1}=\alpha_{2}=\frac{\gamma-1}{4}$)

Using interpolation inequality, for any $0<\alpha_{1}<\gamma$ we have
\begin{align} \label{IEB1}
\|\bar{X}^m_s\|_{\alpha_{1}}
\leq C | \bar{X}^m_s |^{\frac{\gamma-\alpha_{1}}{\gamma}}
\|\bar{X}^m_s\|_{\gamma}^{\frac{\alpha_{1}}{\gamma}},
\end{align}
and for any $0<\alpha_{2}+1<\gamma$ we have
\begin{align} \label{IEB2}
\|\bar{X}^m_s\|_{\alpha_{2}+1}
\leq C |\bar{X}^m_s|^{\frac{\gamma-\alpha_{2}-1}{\gamma}}
\|\bar{X}^m_s\|_{\gamma}^{\frac{\alpha_{2}+1}{\gamma}}.
\end{align}
It follows from $\eref{IEB1}$ and $\eref{IEB2}$ , by Minkowski's inequality we obtain for any $p\geq 1$,
\begin{eqnarray}
\!\!\!\!\!\!\!\!&&\left[\mathbb{E}\Big\|\int^t_0e^{(t-s)A}B(\bar{X}^m_s)ds\Big\|_{\gamma}^{p}\right]^{1/p} \nonumber\\
\leq\!\!\!\!\!\!\!\!&& C
\int^t_0 (t-s)^{\frac{-\alpha_{3}-\gamma}{2}}
\left[\EE| \bar{X}^m_s |^{\frac{p(2\gamma-\alpha_{1}-\alpha_{2}-1)}{\gamma}}
\|\bar{X}^m_s\|_{\gamma}^{\frac{p(\alpha_{1}+\alpha_{2}+1)}{\gamma}}\right]^{1/p}ds \nonumber\\
\leq\!\!\!\!\!\!\!\!&& C
\int^t_0 (t-s)^{\frac{-\alpha_{3}-\gamma}{2}}
\Big(\left[\mathbb{E} |\bar{X}^m_s |^{\frac{p(2\gamma-\alpha_{1}-\alpha_{2}-1)}{\gamma-\alpha_{1}-\alpha_{2}-1}}\right]^{1/p}
+\left[\mathbb{E}\|\bar{X}^m_s\|_{\gamma}^{p}\right]^{1/p}\Big)ds.
\nonumber
\end{eqnarray}
Then by \eref{A.1.0}, there exists $d>1$ such that
\begin{align} \label{HolderNorm B}
\left[\mathbb{E}\Big\|\int^t_0e^{(t-s)A}B(\bar{X}^m_s)ds\Big\|_{\gamma}^{p}\right]^{1/p}
\leq \int^t_0(t-s)^{-\frac{\alpha_{3}+\gamma}{2}}\left[\mathbb{E}\|\bar{X}^m_s\|_{\gamma}^{p}\right]^{1/p}ds+C_T(1+|x|^d).
\end{align}

According to $\eref{P3}$, we obtain for any $\gamma\in(1,3/2)$,
\begin{eqnarray}
\left[\mathbb{E}\Big\|\int^t_0e^{(t-s)A}\bar{F}(\bar{X}^m_s)ds\Big\|_{\gamma}^{p}\right]^{1/p} \leq\!\!\!\!\!\!\!\!&&
C\Big[
\int^t_0 (t-s)^{-\frac{\gamma}{2}}\left[\EE|\bar{F}(\bar{X}^m_s)|^p\right]^{1/p}ds
\nonumber\\
\leq\!\!\!\!\!\!\!\!&&
C\Big[
\int^t_0 (t-s)^{-\frac{\gamma}{2}}\left[\EE(1+|\bar{X}^m_s|^p)\right]^{1/p}ds
\nonumber\\
\leq\!\!\!\!\!\!\!\!&&C_{T}(1+ |x|). \label{BoundF}
\end{eqnarray}

Note that $\int^t_0 (-A)^{\frac{\gamma}{2}} e^{(t-s)A}\sqrt{Q_1}d\bar{W}^{1,m}_s\sim N(0, \tilde{Q}^m_t)$,  which is a Gaussian random variable with mean zero and covariance operator given by
$$
\tilde{Q}^m_t x=\int^t_0 e^{rA}(-A)^{\gamma}Q^m_1 e^{rA}xdr,\quad x\in H_m.
$$
Then for any $p\geq1$, $s>0$, we follow the proof of \cite[Corollary 2.17]{Daz1} to obtain
\begin{align}
\EE\left\|\int^t_0 e^{(t-s)A}\sqrt{Q_1}d\bar{W}^{1}_s\right\|^{p/2}_{\gamma}\leq &\   C[\text{Tr}\tilde{Q}^m_s]^p
= C\left(\Sigma^m_{k=1}\int^t_0e^{-2r\lambda_k}\lambda^{\gamma}_k\alpha_{1,k} dr\right)^{p/2} \nonumber\\
\leq & C\left(\Sigma_{k}\lambda^{\gamma-1}_k\alpha_{1,k}
\int^{2t\lambda_k}_0 e^{-r}dr\right)^{p/2}\nonumber\\
\leq& C(\Sigma_{k}\lambda^{\gamma-1}_k\alpha_{1,k})^{p/2} <\infty,\label{F4.4}
\end{align}
where the last inequality comes from the condition \eref{Noise}.

Finally, combining \eref{HolderNorm B}-\eref{F4.4}, by using Lemma \ref{Gronwall 2}, it is easy to see \eqref{A.1.2}  holds. The proof is complete.
\end{proof}

\begin{remark}
Note that from Lemma \ref{BarXgamma}, using interpolation inequality,
we get that for any  $T>0$, $p\geq 1$, $\gamma_1\in (0, 1]$, $\gamma\in (1,3/2)$ and $x\in H^{\theta}$ with $\theta\in[0,1]$,
there exist constants $d>1$ and $C_{T}>0$ such that for any $t\in (0, T]$,
\begin{align} \label{BarX1}
\sup_{m\in\mathbb{N}_{+}}\left[\EE\|\bar{X}^m_t\|^p_{\gamma_1}\right]^{1/p}\leq
C\left[\EE\left(|\bar{X}^m_t|^{\frac{(\gamma-\gamma_1)p}{\gamma}}
\|\bar{X}^m_t\|^{\frac{\gamma_1 p}{\gamma}}_{\gamma}\right)\right]^{1/p}
\leq C_{T}t^{-\frac{(\gamma-\theta) \gamma_1 }{2\gamma}}(1+\|x\|^{\frac{\gamma_1 }{\gamma}}_{\theta}+|x|^d).
\end{align}
\end{remark}

\vspace{0.2cm}
In order to estimate $\EE\|\bar{X}^m_t\|^p_2$, we need the following lemma.
\begin{lemma}\label{Lemma 3.3}
For any initial value $x\in H^{\theta}$ with $\theta\in[0,1]$, $\gamma\in (1,3/2)$, $0<s<t\leq T$, $p\geq1$,
there exist constants $d>1$ and $C_{T}>0$ such that
\begin{align}
\sup_{m\in\mathbb{N}_{+}}\left[\EE\|\bar{X}^m_t-\bar{X}^m_s\|^p_{1}\right]^{1/p}
\leq C_T(t-s)^{\frac{\gamma-1}{2}}s^{-\frac{\gamma-\theta}{2}}(1+|x|^d)(1+\|x\|_{\theta}). \label{COXT}
\end{align}
\end{lemma}

\begin{proof}
For any $T>0$ with $0<s<t\leq T$, we write
\begin{align} \label{bar X Contin}
\bar{X}^m_t-\bar{X}^m_s =& \ \big (e^{(t-s)A}-I \big)\bar{X}^m_s
+\int_{s}^{t}e^{(t-r)A}B^m(\bar{X}^m_r)dr\nonumber\\
&\ +\int_{s}^{t}e^{(t-r)A}\bar{F}^m(\bar{X}^m_r)dr+\int^t_se^{(t-r)A}\sqrt{Q_1}d\bar{W}^{1,m}_r .
\end{align}

Using \eref{P4} and Lemma \ref{BarXgamma}, there exist $d>1$, $C_T>0$ such that for any $\gamma\in (1,3/2)$,
\begin{eqnarray}  \label{bar X Contin 1}
\left[\EE\left\|\big(e^{A(t-s)}-I \big)\bar{X}^m_s\right\|^p_{1}\right]^{1/p}
\leq\!\!\!\!\!\!\!\!&&
C(t-s)^{\frac{\gamma-1}{2}}\left[\EE\|\bar{X}^m_s\|^p_{\gamma} \right]^{1/p} \nonumber\\
\leq\!\!\!\!\!\!\!\!&&
C_T(t-s)^{\frac{\gamma-1}{2}}s^{-\frac{\gamma-\theta}{2}}(1+|x|^d+\|x\|_{\theta}).
\end{eqnarray}

Using \eref{P3} and  Lemma \ref{BarXgamma}, there exist $d>1$, $C_T>0$ such that for any $\gamma\in (1,3/2)$, $\theta\in[0,1]$,
\begin{eqnarray}  \label{bar X Contin 2}
\left[\EE\left\|\int_{s}^{t}e^{(t-r)A}B^m(\bar{X}^m_r)dr\right\|^p_{1}\right]^{1/p}\leq\!\!\!\!\!\!\!\!&&
C\int_{s}^{t}(t-r)^{-\frac{3-\gamma}{2}}\left[\EE\left\|B^m(\bar{X}^m_r)\right\|^p_{-(2-\gamma)}\right]^{1/p}dr  \nonumber\\
\leq\!\!\!\!\!\!\!\!&&
C_T\int_{s}^{t}(t-r)^{-\frac{3-\gamma}{2}}\left[\EE|\bar{X}^m_r|^p\|\bar{X}^m_r\|^{p}_{\gamma}\right]^{1/p}dr  \nonumber\\
\leq\!\!\!\!\!\!\!\!&&
C_T\int_{s}^{t}(t-r)^{-\frac{3-\gamma}{2}}r^{-\frac{\gamma-\theta}{2}}(1+|x|^d)(1+\|x\|_{\theta})dr  \nonumber\\
\leq\!\!\!\!\!\!\!\!&&
C_T(t-s)^{\frac{\gamma-1}{2}}s^{-\frac{\gamma-\theta}{2}}(1+|x|^d)(1+\|x\|_{\theta})
\end{eqnarray}
and
\begin{eqnarray} \label{bar X Contin 3}
\left[\EE\left\|\int_{s}^{t}e^{(t-r)A}\bar{F}^m(\bar{X}^m_r)dr\right\|^p_{1}\right]^{\frac{1}{p}}
\leq\!\!\!\!\!\!\!\!&&
C\int_{s}^{t}[1+(t-r)^{-\frac{1}{2}}]\left[\EE\left|\bar{F}^m(\bar{X}^m_r)\right|^p\right]^{\frac{1}{p}}dr
\nonumber\\
\leq\!\!\!\!\!\!\!\!&&
C\int_{s}^{t}[1+(t-r)^{-\frac{1}{2}}]\left[\EE(1+ | \bar{X}^m_r|^p)\right]^{\frac{1}{p}}dr  \nonumber\\
\leq\!\!\!\!\!\!\!\!&&
C_T(t-s)^{\frac{1}{2}}(1 + |x|).
\end{eqnarray}

By condition \eref{Noise}, it is easy to see that
\begin{eqnarray}
\left[\EE\left\|\int^t_s e^{(t-r)A}\sqrt{Q_1}d\bar{W}^{1,m}_t\right\|^{p}_{1}\right]^{\frac{1}{p}}\leq\!\!\!\!\!\!\!\!&&C\left[\Sigma_{k}\int^{t-s}_0e^{-2r\lambda_k}\lambda_k\alpha_{1,k} dr\right]^{\frac{1}{2}} \nonumber\\
=\!\!\!\!\!\!\!\!&&C\left[\Sigma_{k}\alpha_{1,k}(1-e^{-2(t-s)\lambda_k})\right]^{\frac{1}{2}}\nonumber\\
\leq\!\!\!\!\!\!\!\!&&C(\Sigma_{k}\lambda_k\alpha_{1,k})^{\frac{1}{2}} (t-s)^{\frac{1}{2}} .\label{DW}
\end{eqnarray}

Finally, by combining \eqref{bar X Contin}-\eqref{bar X Contin 3}, we can easily obtain that \eref{COXT} holds. The proof is complete.
\end{proof}

\begin{lemma} \label{Xbar2}
For any $T>0$, $p\geq1$ and $x\in H^{\theta}$ with $\theta\in(0,1]$, there exist constants $d>1$ and $C_{T}$ such that for any $0<t\leq T$,
\begin{eqnarray}
\sup_{m\in\mathbb{N}_{+}}\left[\mathbb{E}\|\bar{X}^m_{t}\|^p_2\right]^{\frac{1}{p}}
\leq C_T t^{-1+\frac{\theta}{2}}(1+|x|^d)(1+\|x\|^2_{\theta}).  \label{barXm2}
\end{eqnarray}
\end{lemma}

\begin{proof}
For $t>0$, we rewrite
\begin{eqnarray}
\bar{X}^m_t
=\!\!\!\!\!\!\!\!&&e^{tA}x^m+\int_{0}^{t}e^{(t-s)A}B^m(\bar{X}^m_t)ds+\int_{0}^{t}e^{(t-s)A}\left[B^m(\bar{X}^m_s)-B^m(\bar{X}^m_t)\right]ds  \nonumber\\
\!\!\!\!\!\!\!\!&&+\int_{0}^{t}e^{(t-s)A}\bar{F}^m(\bar{X}^m_t)ds
+\int_{0}^{t}e^{(t-s)A}\left[\bar{F}^m(\bar{X}^m_s)-\bar{F}^m(\bar{X}^m_t)\right]ds \nonumber\\
&&+\int_{0}^{t}e^{(t-s)A}\sqrt{Q_1}d\bar{W}^{1,m}_s\nonumber\\
:=\!\!\!\!\!\!\!\!&&I_{1}+I_{2}+I_{3}+I_{4}+I_{5}+I_6. \nonumber
\end{eqnarray}

For the term $I_{1}$, using \eqref{P3} we have
\begin{align} \label{ABarX1}
\|e^{tA}x^m\|_2
\leq C t^{-1+\frac{\theta}{2}}\|x\|_{\theta}.
\end{align}

For the term $I_{2}$, by \eref{BarX1} and \eref{A.1.2}, we have for any $p\geq 1$, $\gamma\in (1,3/2)$,
\begin{eqnarray} \label{ABarX2}
\left[\EE\|I_{2}\|^p_2\right]^{1/p}=\!\!\!\!\!\!\!\!&&\left[\EE\left\|(e^{tA}-I)B^m(\bar{X}^m_t)\right\|^p\right]^{1/p}\nonumber\\
 \leq\!\!\!\!\!\!\!\!&&C\left[\EE\left\|\bar{X}^m_t\right\|_{1/2}^{p}\left\|\bar{X}^m_t\right\|_{\gamma}^{p}\right]^{1/p}  \nonumber\\
\leq\!\!\!\!\!\!\!\!&&C t^{-\frac{1}{4}-\frac{\gamma-\theta}{2}}(1+|x|^d)(1+\|x\|_{\theta})\nonumber\\
\leq\!\!\!\!\!\!\!\!&& C_T t^{-1+\theta/2}(1+|x|^d)(1+\|x\|_{\theta}).
\end{eqnarray}

For the term $I_{3}$, by Lemma \ref{Property B1}, we obtain
\begin{eqnarray*}
\|I_{3}\|_2\leq\!\!\!\!\!\!\!\!&&C\left|\int_{0}^{t}(-A)e^{(t-s)A}\left[B^m(\bar{X}^m_s-\bar{X}^m_t,\bar{X}^m_t)+B^m(\bar{X}^m_s, \bar{X}^m_s-\bar{X}^m_t)\right]ds\right| \nonumber\\
\leq\!\!\!\!\!\!\!\!&&C\int_{0}^{t}(t-s)^{-1}|B^m(\bar{X}^m_s-\bar{X}^m_t,\bar{X}^m_t)|ds  \nonumber\\
&&+C\int_{0}^{t}(t-s)^{-1}|B^m(\bar{X}^m_s, \bar{X}^m_s-\bar{X}^m_t)|ds. \nonumber
\end{eqnarray*}
According to Minkowski's inequality, interpolation inequality, \eref{A.1.0}, \eref{BarX1} and \eref{COXT}, it follows that for small enough $\delta>0$
\begin{eqnarray}  \label{ABarX3}
\left[\mathbb{E}\|I_{3}\|^p_2\right]^{1/p}
\leq\!\!\!\!\!\!\!\!&&
C\int_{0}^{t}(t-s)^{-1}\left[\mathbb{E}\|\bar{X}^m_t-\bar{X}^m_s\|^{2p}_{1}\right]^{\frac{1}{2p}}\left[\EE\|\bar{X}^m_t\|^{\frac{1}{2p}}_{1}\right]^{\frac{1}{2p}}ds \nonumber\\
&&+C\int_{0}^{t}(t-s)^{-1}\left[\mathbb{E}\|\bar{X}^m_s\|^{2p}_{1/2+\delta}\right]^{\frac{1}{2p}}\left[\EE\|\bar{X}^m_s-\bar{X}^m_t\|^{2p}_{1}\right]^{\frac{1}{2p}}ds\nonumber\\
\leq\!\!\!\!\!\!\!\!&&
C(1+|x|^d)(1+\|x\|_{\theta})\int_{0}^{t}(t-s)^{-1}(t-s)^{\frac{\gamma-1}{2}}s^{-\frac{\gamma-\theta}{2}}t^{-\frac{1}{2}}ds \nonumber\\
&&+C(1+|x|^d)(1+\|x\|_{\theta})\int_{0}^{t}(t-s)^{-1}(t-s)^{\frac{\gamma-1}{2}}s^{-\frac{\gamma-\theta}{2}}s^{-\frac{1}{4}-\frac{\delta}{2}}ds\nonumber\\
\leq\!\!\!\!\!\!\!\!&& C_T t^{-1+\frac{\theta}{2}}(1+|x|^d)(1+\|x\|_{\theta}).
\end{eqnarray}

For the term $I_{4}$, we have
\begin{align} \label{ABarX4}
\left[\mathbb{E}\|I_{4}\|^p_2\right]^{1/p}=
\left(\mathbb{E}\left|(e^{tA}-I)\bar{F}^m(\bar{X}^m_t)\right|^p\right)^{1/p}
\leq  C\left(1+\left(\mathbb{E}|\bar{X}^m_t|^p\right)^{1/p}\right)
\leq  C(1+|x| ).
\end{align}

For the term $I_{5}$, using Minkowski's inequality and \eref{COXT}, we obtain
\begin{eqnarray} \label{ABarX5}
\left[\mathbb{E} \| I_{5}\|^p_2\right]^{1/p}
\leq\!\!\!\!\!\!\!\!&&
C\int_{0}^{t}(t-s)^{-1}\left(\mathbb{E}\left\|\bar{X}^m_t-\bar{X}^m_s\right\|^{p}_1\right)^{\frac{1}{p}}ds \nonumber\\
\leq\!\!\!\!\!\!\!\!&&
C t^{-\frac{1}{2}}(1+|x|^d) .
\end{eqnarray}

For the term $I_{6}$, similar as we did in \eref{F4.4}, we easily have
\begin{align}
\left[\mathbb{E} \| I_{6}\|^p_2\right]^{1/p}
\leq C(\Sigma_{k}\lambda_k\alpha_{1,k})^{1/2} <\infty .\label{ABarX6}
\end{align}

Combining \eref{ABarX1}-\eref{ABarX6}, we easily obtain the desired result. The proof is complete.
\end{proof}

\subsection{The a priori estimates of the solution $(X^{m,\vare}_t, Y^{m,\vare}_t)$}
\begin{lemma} \label{L6.1}
For any $x,y\in H$, $p\geq 1$ and $T>0$, there exist constants $C_{p,T},\ C_{T}>0$ such that the solution $(X^{m,\vare}_t, Y^{m,\vare}_t)$ of system \eref{main equation} satisfies
\begin{eqnarray}
&&\sup_{m\in\mathbb{N}_{+}}\mathbb{E}\left(\sup_{t\in [0,T]}|X_{t}^{m,\vare}|^p\right) \leq  C_{p,T}(1+ |x|^p),\quad \forall \vare>0;\label{AXvare}\\
&&\sup_{m\in\mathbb{N}_{+}}\sup_{t\geq 0}\mathbb{E} |Y_{t}^{m,\vare} |^p\leq C_{p,T}(1+|y|^p ),\quad \forall \vare>0;\label{AYvare}\\
&&\sup_{m\in\mathbb{N}_{+}}\mathbb{E} \left[\sup_{t\in [0,T]}|Y_{t}^{m,\vare} |^p\right]\leq \frac{C_{p,T}(1+|y|^p )}{\vare},\quad \forall \vare\leq T.\label{ASupYvare}
\end{eqnarray}
\end{lemma}
\begin{proof}
Note the condition \eref{F2.21} holds, by a minor revision as in the proof of \cite[Lemma 3.1]{DSXZ}, it is easy to see \eref{AXvare} and \eref{AYvare} hold. Thus we only prove \eref{ASupYvare} here.

For fixed $\vare>0$, define $\tilde{Y}^{\vare}_{t}:=Y^{\vare}_{t\vare}$. It is easy to check that the process $\{\tilde{Y}^{\vare}_{t}\}_{t\geq 0}$ satisfies
\begin{eqnarray*}
\left\{ \begin{aligned}
&d\tilde{Y}^{\vare}_{t}=\left[A\tilde{Y}^{\vare}_{t}+G(X^{\vare}_{t\vare},\tilde{Y}^{\vare}_{t})\right]dt+\sqrt{Q_2}d \tilde W_{t}^2,\\
&\tilde{Y}^{\vare}_{0}=y,
\end{aligned} \right.
\end{eqnarray*}
where $\{\tilde{W}_{t}^2:=\vare^{-1/2}W_{t\vare}^2\}_{t\geq 0}$ is also a cylindrical Wiener process, which has the same distribution with $\{W_t^2\}_{t\geq 0}$.

By It\^{o}'s formula, we have
\begin{eqnarray}
|\tilde{Y}^{\vare}_{t}|^p=\!\!\!\!\!\!\!\!&&|y|^p+p\int^{t}_{0}|\tilde{Y}^{\vare}_{s}|^{p-2}\langle \tilde{Y}^{\vare}_{s},A\tilde{Y}^{\vare}_{s}+G(X^{\vare}_{s\vare},\tilde{Y}^{\vare}_{s})\rangle ds\nonumber\\
\!\!\!\!\!\!\!\!&&+p\int^{t}_{0}|\tilde{Y}^{\vare}_{s}|^{p-2}\langle \tilde{Y}^{\vare}_{s},\sqrt{Q_2}d\bar{W}_s^{2,m}\rangle\nonumber\\
&&+\frac{p}{2}\int_{0} ^{t}|\tilde{Y}^{\vare}_{s}|^{p-2}Tr(Q_2)ds\nonumber\\
&&+\frac{p(p-2)}{2}\int_{0}^{t}|\tilde{Y}^{\vare}_{s}|^{p-4}|\sqrt{Q_2} \tilde{Y}^{\vare}_{s}|^2ds.\label{Ito formula}
\end{eqnarray}

Then by \eref{F2.21}, it is easy to see
\begin{eqnarray*}
\langle y, Ay+G^m(x,y)\rangle\leq -\lambda_1|y|^2+|y|(L_{G}|y|+C)\leq -\frac{\lambda_1-L_{G}}{2}|y|^2+C.
\end{eqnarray*}
Applying Burkholder-Davies-Gundy's inequality and \eref{AYvare}, we have for any $T\geq 1$,
\begin{eqnarray}
\mathbb{E}\left[\sup_{t\in [0,T]}|\tilde{Y}^{\vare}_{t}|^p\right]\nonumber
\leq\!\!\!\!\!\!\!\!&&|y|^p+C_p\mathbb{E}\int^{T}_{0}|\tilde{Y}^{\vare}_{s}|^{p-2}ds+C_p\mathbb{E}\Big[\int^{T}_{0}|\tilde{Y}^{\vare}_{s}|^{2p-2}ds\Big]^{1/2}\nonumber\\
\leq\!\!\!\!\!\!\!\!&&C_p(1+|y|^p)T,\label{I1}
\end{eqnarray}
This shows that
\begin{eqnarray*}
\EE\left[\sup_{t\in [0,T]}|Y^{\vare}_{t}|^p\right]=\EE\left[\sup_{t\in [0, T/\vare]}|\tilde{Y}^{\vare}_{t}|^p\right]\leq \frac{C_{p,T}(1+|y|^p)}{\vare}.
\end{eqnarray*}
The proof is complete.
\end{proof}

Note the \eref{F2.21} holds, we can easily obtain the following lemma. Since the proofs almost follows the same steps in Lemmas \ref{BarXgamma} and \ref{Lemma 3.3}, we omit the proofs.
\begin{lemma} \label{Xgamma}
(1) For any $T>0$, there exists a constant $C_T>0$ such that for any $x\in H$, $\vare\in(0,1]$ and $\delta<\frac{\lambda_1-2L_{F}}{4\|Q_1\|}$, we have
\begin{align}
\sup_{m\in\mathbb{N}_{+}}\EE\left(e^{\delta\sup_{t\leq T}|X^{m,\vare}_t|^2+\delta\int^T_0\|X^{m,\vare}_s\|^2_1ds}\right)\leq C_T(1+|x|)e^{\delta |x|^2}.\label{X11}
\end{align}
(2) For any $T>0$, $p\geq 1$, $\gamma\in(1,3/2)$ and $x\in H^{\theta}$ with $\theta\in [0, \gamma]$,
there exist constants $d>1$ and $C_{T}>0$ such that for any $t\in (0, T]$, we have
\begin{align}
\sup_{m\in\mathbb{N}_{+}}\left[\EE\|X^{m,\vare}_t\|^p_{\gamma}\right]^{1/p}\leq C_{T} t^{-\frac{(\gamma-\theta) }{2}}(\|x\|_{\theta}+1). \label{X12}
\end{align}
(3) For any $x\in H^{\theta}$ with $\theta\in[0,1]$, $\gamma\in (1,3/2)$, $0<s<t\leq T$, $p\geq1$,
there exist constants $d>1$ and $C_{T}>0$ such that
\begin{eqnarray*}
\sup_{m\in\mathbb{N}_{+}}\left[\EE\|X^{m,\vare}_t-X^{m,\vare}_s\|^p_{1}\right]^{1/p}
\leq C_T(t-s)^{\frac{\gamma-1}{2}}s^{-\frac{\gamma-\theta}{2}}(1+|x|^d)(1+\|x\|_{\theta}).
\end{eqnarray*}
\end{lemma}
\begin{remark}
Using interpolation inequality,
we get that for any $p\geq1$, $\gamma_1\in (0, 1]$, $\gamma\in (1,3/2)$, $T>0$ and $x\in H^{\theta}$ with $\theta\in[0,1]$,
there exist constants $d>1$ and $C_{T}>0$ such that for any $t\in (0, T]$, $\vare\in(0,1]$,
\begin{eqnarray} \label{X1}
&&\sup_{m\in\mathbb{N}_{+}}\left[\EE\|X^{m,\vare}_t\|^p_{\gamma_1}\right]^{1/p}
\leq C_{T}t^{-\frac{(\gamma-\theta) \gamma_1 }{2\gamma}}(1+\|x\|^{\frac{\gamma_1 }{\gamma}}_{\theta}+|x|^d).
\end{eqnarray}
\end{remark}

For the purpose of achieving the optimal convergence order, we use another strategy to estimate $\left[\mathbb{E}\|X^{m,\vare}_t\|^p_2\right]^{\frac{1}{p}}$, which is different from the ones used in Lemma \ref{Xbar2}.
\begin{lemma}
For any $p>1$, $x\in H^{\theta}$ with $\theta\in(0,1)$,  $0<t\leq T$, there exists a constant $C_{T}$ such that for any $\vare\in (0,1]$,
\begin{align}
\left[\mathbb{E}\|X^{m,\vare}_t\|^p_2\right]^{\frac{1}{p}}
\leq C_T (t^{-1+\frac{\theta}{2}}+t^{\frac{\alpha-\beta}{2}})(1+|x|^d+\|x\|_{\theta}+|y|). \label{Xvare2}
\end{align}
\end{lemma}
\begin{proof}
Note that for  any $T>0$ with $t\in (0,T]$, there exist $C_T>0$ and $d>1$ such that
\begin{eqnarray*}
X^{m,\vare}_t=\!\!\!\!\!\!\!\!&&e^{tA}x+\int_{0}^{t}e^{(t-s)A}B^m(X^{\vare}_{t})ds+\int_{0}^{t}e^{(t-s)A}\left[B^m(X^{m,\vare}_s)-B^m(X^{m,\vare}_{t})\right]ds  \nonumber\\
\!\!\!\!\!\!\!\!&&+\int_{0}^{t}e^{(t-s)A}F^m(X^{m,\vare}_s,Y^{m,\vare}_{s})ds+\int_{0}^{t}e^{(t-s)A}\sqrt{Q_1}d\bar{W}^{1,m}_s.
\end{eqnarray*}
Then by following the same argument in the proof of Lemma \ref{Xbar2}, we get for any $p>1$,
\begin{eqnarray}
\left[\mathbb{E}\|X^{m,\vare}_t\|^p_2\right]^{\frac{1}{p}}
\leq \!\!\!\!\!\!\!\!&&C t^{-1+\frac{\theta}{2}}(1+|x|^d)(1+\|x\|_{\theta})+\left[\EE\left\|\int_{0}^{t}e^{(t-s)A}F(X^{m,\vare}_s,Y^{m,\vare}_{s})ds\right\|^p_2\right]^{1/p}\nonumber\\
\leq \!\!\!\!\!\!\!\!&&C t^{-1+\frac{\theta}{2}}(1+|x|^d)(1+\|x\|_{\theta})+\int_{0}^{t}(t-s)^{-1+\alpha/2}\left[\EE \left\|F(X^{m,\vare}_s,Y^{m,\vare}_{s})\right\|^p_{\alpha}\right]^{1/p}ds\nonumber\\
\leq \!\!\!\!\!\!\!\!&&C t^{-1+\frac{\theta}{2}}(1+|x|^d)(1+\|x\|_{\theta})\nonumber\\
&&+C\int_{0}^{t}(t-s)^{-1+\alpha/2}\left[\EE (1+\|X^{m,\vare}_s\|^p_{\beta}+\|Y^{m,\vare}_{s}\|^p_{\beta})\right]^{1/p}ds,\label{F3.42}
\end{eqnarray}
where $\alpha, \beta\in (0,1)$ are the constants in assumption \ref{A3}.

Note that by \eref{BarX1}, for any $t\in (0, T]$, we have
\begin{eqnarray}
\left[\EE\|X^{m,\vare}_t\|^p_{\beta}\right]^{1/p}\leq C_T t^{-\frac{\beta}{2}}(1+|x|^d)\label{Xbeta}
\end{eqnarray}
 and
\begin{eqnarray}
\left[\EE\|Y^{m,\varepsilon}_t\|^p_{\beta}\right]^{1/p}
\leq\!\!\!\!\!\!\!\!&&\|e^{tA/\vare}y\|_{\beta}+\frac{1}{\vare}\int^t_0\left[\EE\|e^{\frac{(t-s)A}{\vare}}G^m(X^{m,\vare}_s,Y^{m,\vare}_s)\|^p_{\beta}\right]^{1/p} ds\nonumber\\
&&+\left[\EE\left\|\frac{1}{\sqrt{\vare}}\int^t_0 e^{\frac{(t-s)A}{\vare}}\sqrt{Q_2}d\bar W^{2,m}_s\right\|^p_{\beta}\right]^{1/p}\nonumber\\
\leq\!\!\!\!\!\!\!\!&&C\left(\frac{t}{\vare}\right)^{-\frac{\beta}{2}}|y|+\frac{C}{\vare}\int^t_0\left(\frac{t-s}{\vare}\right)^{-\frac{\beta}{2}}e^{-\frac{(t-s)\lambda_1}{2\vare}}\left[\EE(1+|Y^{m,\varepsilon}_s|^p)\right]^{1/p}ds\nonumber\\
&&+C \text{Tr}Q_2\left[\int^t_0 \frac{1}{\vare}\left(\frac{t-s}{\vare}\right)^{-\beta}e^{-\frac{(t-s)\lambda_1}{\vare}}ds\right]^{1/2}\nonumber\\
\leq\!\!\!\!\!\!\!\!&&C_T t^{-\frac{\beta}{2}}(1+|y|).\label{Ybeta}
\end{eqnarray}
By \eref{F3.42}-\eref{Ybeta}, it is easy to see that
\begin{align}
\left[\mathbb{E}\|X^{m,\vare}_t\|^p_2\right]^{\frac{1}{p}}
\leq C_T(t^{-1+\frac{\theta}{2}}+t^{\frac{\alpha-\beta}{2}})(1+|x|^d+\|x\|_{\theta}+|y|).  \nonumber
\end{align}
The proof is complete.
\end{proof}

\section{The proof of strong convergence}

In this section, we indent to give the detailed proofs of Theorem \ref{main result1}. The main technique is based on the well-known Poisson equation, i.e, considering the following  equation:
\begin{equation}
-\mathscr{L}^m_{2}(x)\Phi_m(x,y)=F^m(x,y)-\bar{F}^m(x),\quad x, y\in H_m.\label{ID2}
\end{equation}
where $\mathscr{L}^m_{2}(x)$ is the infinitesimal generator of the transition semigroup of the finite dimensional frozen equation \eref{Ga FZE}, i.e.,
\begin{eqnarray}
\mathscr{L}^m_{2}(x)\Phi_m(x,y):=\!\!\!\!\!\!\!\!&&D_y\Phi_m(x,y)\cdot [Ay+G^m(x,y)]\nonumber\\
&&+\frac{1}{2}\sum^{m}_{k=1}\left[D_{yy}\Phi_m(x,y)\cdot(\sqrt{Q_2} e_k,  \sqrt{Q_2} e_k)\right].\label{L_2}
\end{eqnarray}

Note that  \eref{F2.21} holds, by a minor revision in \cite[Proposition 3.4]{GSX}, we can obtain the following proposition, which plays an important role in the proofs. Since the proofs almost the same, we do not give the detailed proof here. See \cite{B3} for similar result of the regularity of the solution of Poisson equation.
\begin{proposition}
Poisson equation \eref{ID2} admits a solution
 \begin{eqnarray}
\Phi_m(x,y):=\int^{\infty}_{0}\left[\EE F^m(x,Y^{x,y,m}_t)-\bar{F}^m(x)\right]dt.\label{SPE}
\end{eqnarray}
Moreover, there exists $C>0$ such that for any $y,h,k\in H_m$,
\begin{eqnarray}
&&\sup_{x\in H_m,m\geq1}|\Phi_m(x,y)|\leq C(1+|y|), \label{E0}\\
&&\sup_{x,y\in H_m,m\geq 1}|D_y \Phi_m(x,y)\cdot h|\leq C|h|, \label{E1}\\
&&\sup_{x\in H_m,m\geq 1}|D_x \Phi_m(x,y)\cdot h|\leq C(1+|y|)|h|,\label{E2}\\
&& \sup_{x\in H_m,m\geq 1}| D_{xx}\Phi_m(x, y)\cdot(h,k)|\leq C(1+|y|)|h|\|k\|_{\tau},\label{E3}
\end{eqnarray}
where $\tau\in (0,2)$ is the constant in \ref{A2}.
\end{proposition}

\vspace{0.2cm}
Now, we give a position to prove first main result.

\textbf{Proof of Theorem \ref{main result1}}.
\vspace{2mm}
It is easy to see that for any $T>0$ and $m\in \mathbb{N}_{+}$,
\begin{eqnarray*}
\EE\left(\sup_{t\in [0, T]}|X_{t}^{\vare}-\bar{X}_{t}|^p\right)\leq\!\!\!\!\!\!\!\!&&C_p\EE\left(\sup_{t\in [0, T]}|X_{t}^{m,\vare}-X_{t}^{\vare}|^p\right)+C_p\EE\left(\sup_{t\in [0, T]}|\bar{X}^m_{t}-\bar{X}_{t}|^p\right)\\
&&+C_p\EE\left(\sup_{t\in [0, T]}|X_{t}^{m,\vare}-\bar{X}^m_{t}|^p\right).
\end{eqnarray*}

By the finite approximations \eref{FA1} and \eref{FA2} in the Appendix, it is sufficient to prove that for any initial value $(x,y)\in H^{\theta}\times H$ with $\theta\in (0,1]$, $p\geq 2$, $T>0$, there exist $d>1, C_{p,T}>0$ such that
\begin{eqnarray}
\sup_{m\in\mathbb{N}_{+}}\EE\left(\sup_{t\in [0, T]}|X_{t}^{m,\vare}-\bar{X}^m_{t}|^p\right)\leq C_{p,T} e^{\delta |x|^2}(1+|x|^d+\|x\|^{p}_{\theta}+|y|^{p})\vare^{p/2},\label{FR}
\end{eqnarray}
where $\delta$ is a small enough constant. For the reason of technique, we first consider the case that $T$ is small enough. Then we extend to arbitrary $T>0$ by Markov property.

\vspace{0.1cm}
Putting $Z^{m,\varepsilon}_t:=X_{t}^{m,\vare}-\bar{X}^m_{t}$, thus it follows
\begin{align*}
d Z^{m,\varepsilon}_t=AZ^{m,\varepsilon}_t dt+\big[B^m(X^{m,\varepsilon}_t)-B^m(\bar{X}^m_{t})\big]dt
+\big[F^m(X^{m,\varepsilon}_t, Y^{m,\varepsilon}_t)-\bar{F}^m(\bar{X}^m_{t})\big]dt,\quad Z^{m,\varepsilon}_0=0.
\end{align*}
Multiplying both sides by $2Z^{m,\varepsilon}_t$, then by \eref{P3}, \eref{BP2} and Young's inequality, we get
\begin{eqnarray*}
|Z^{m,\varepsilon}_t|^2=\!\!\!\!\!\!\!\!&&2\int_{0}^{t}\langle AZ^{m,\varepsilon}_s, Z^{m,\varepsilon}_s\rangle ds
+2\int_{0}^{t}\langle B^m(X^{m,\varepsilon}_s)-B^m(\bar{X}^m_{s}), Z^{m,\varepsilon}_s\rangle ds\nonumber\\
&&+2\int_{0}^{t}\langle F^m(X^{m,\varepsilon}_s, Y^{m,\varepsilon}_s)-\bar F^m(\bar{X}^m_{s}), Z^{m,\varepsilon}_s\rangle ds\nonumber\\
\leq\!\!\!\!\!\!\!\!&&-2\int_{0}^{t}\|Z^{m,\varepsilon}_s\|^2_1 ds+C\int_{0}^{t}|Z^{\varepsilon}_s|\|Z^{\varepsilon}_s\|_1 (\|X^{m,\varepsilon}_s\|_{\gamma}+\|\bar{X}^m_s\|_{\gamma}) ds+C\int^t_0 |Z^{\varepsilon}_s|^2ds\\
&&+2\int_{0}^{t}\langle F^m(X^{m,\varepsilon}_s, Y^{m,\varepsilon}_s)-\bar{F}^m(X^{m,\varepsilon}_{s}), Z^{m,\varepsilon}_s\rangle ds\nonumber\\
\leq\!\!\!\!\!\!\!\!&&C\int_{0}^{t}|Z^{m,\varepsilon}_s|^2 (\|X^{m,\varepsilon}_s\|^2_{\gamma}+\|\bar{X}^m_s\|^2_{\gamma}+1) ds+2\left|\int_{0}^{t}\langle F^m(X^{m,\varepsilon}_s, Y^{m,\varepsilon}_s)-\bar{F}^m(X^{m,\varepsilon}_{s}), Z^{m,\varepsilon}_s\rangle ds\right|,
\end{eqnarray*}
where $\gamma\in (1/2,1)$. Then by Gronwall's inequality, it arrives
\begin{eqnarray*}
\sup_{t\in [0,T]}|Z^{m,\varepsilon}_t|^2
\leq\!\!\!\!\!\!\!\!&&2\sup_{t\in [0,T]}\left|\int_{0}^{t}\langle F(X^{m,\varepsilon}_s, Y^{m,\varepsilon}_s)-\bar{F}(X^{m,\varepsilon}_{s}), Z^{m,\varepsilon}_s\rangle ds\right|e^{C\int_{0}^{T}(\|X^{m,\varepsilon}_s\|^2_{\gamma}+\|\bar{X}^m_s\|^2_{\gamma}+1) ds}.
\end{eqnarray*}

For any $p\geq 4$, by taking $p/2$-th moment and expectation on both sides, we get
\begin{eqnarray*}
\EE\left[\sup_{t\in [0,T]}|Z^{m,\varepsilon}_t|^p\right]
\leq\!\!\!\!\!\!\!\!&&C_p\left[\EE\left(\sup_{t\in [0,T]}\left|\int_{0}^{t}\langle F(X^{m,\varepsilon}_s, Y^{m,\varepsilon}_s)-\bar{F}(X^{m,\varepsilon}_{s}), Z^{m,\varepsilon}_s\rangle ds\right|^p\right)\right]^{1/2}\nonumber\\
&&\quad\cdot\left[\EE e^{C_p\int_{0}^{T}(\|X^{m,\varepsilon}_s\|^2_{\gamma}+\|\bar{X}^m_s\|^2_{\gamma}+1) ds}\right]^{1/2}.
\end{eqnarray*}

Using interpolation inequality and Young's inequality,  there exists a small enough $\delta<\frac{\lambda_1-2 L_{F}}{8\|Q_1\|}$ such that
\begin{eqnarray*}
&&C_p\int_{0}^{T}\left(\|X^{m,\varepsilon}_s\|^2_{\gamma}+\|\bar{X}^m_s\|^2_{\gamma}\right)ds \\ \leq\!\!\!\!\!\!\!\!&&C_p\int_{0}^{T}\left(|X^{m,\varepsilon}_s|^{2(1-\gamma)}\|X^{m,\varepsilon}_s\|^{2\gamma}_{1}+|\bar{X}^{m}_s|^{2(1-\gamma)}\|\bar{X}^{m}_s\|^{2\gamma}_{1}\right)ds\\
\leq\!\!\!\!\!\!\!\!&&C_p\int_{0}^{T}\left(|X^{m,\varepsilon}_s|^{2}+|\bar{X}^{m}_s|^{2}\right)ds+\delta\int^T_0\left(\|X^{m,\varepsilon}_s\|^{2}_{1}+\|\bar{X}^{m}_s\|^{2}_{1}\right)ds\nonumber\\
\leq\!\!\!\!\!\!\!\!&&C_p T\left(\sup_{s\in [0,T]}|X^{m,\varepsilon}_s|^{2}+\sup_{s\in [0,T]}|\bar{X}^{m}_s|^{2}\right)+\delta\int^T_0\left(\|X^{m,\varepsilon}_s\|^{2}_{1}+\|\bar{X}^{m}_s\|^{2}_{1}\right)ds.
\end{eqnarray*}
Then  there exists a small enough $T>0$ satisfying $C_p T\leq \delta$, by \eref{A.1.1} and \eref{X11} we have
\begin{eqnarray*}
\EE\left[\sup_{t\in [0,T]}|Z^{m,\varepsilon}_t|^p\right]
\leq\!\!\!\!\!\!\!\!&&C_T(1+|x|) e^{\delta |x|^2}\left[\EE\sup_{t\in [0,T]}\left|\int_{0}^{t}\langle F(X^{m,\varepsilon}_s, Y^{m,\varepsilon}_s)-\bar{F}(X^{m,\varepsilon}_{s}), Z^{m,\varepsilon}_s\rangle ds\right|^p\right]^{1/2}.
\end{eqnarray*}

Recall the solution $\Phi_m(x,y)$ of Poisson equation \eref{ID2}, note that
$$
4\langle\Phi_m(X^{m,\vare}_{t},Y^{m,\vare}_{t}), Z^{m,\varepsilon}_t\rangle=|\Phi_m(X^{m,\vare}_{t},Y^{m,\vare}_{t})+ Z^{m,\varepsilon}_t|^2-|\Phi_m(X^{m,\vare}_{t},Y^{m,\vare}_{t})- Z^{m,\varepsilon}_t|^2.
$$
Then applying It\^o's formula we get
\begin{eqnarray}
\langle\Phi_m(X^{m,\vare}_{t},Y^{m,\vare}_{t}), Z^{m,\varepsilon}_t\rangle=\!\!\!\!\!\!\!\!&&\int^t_0 \left\langle\Phi_m(X^{m,\vare}_{s},Y^{m,\vare}_{s}), AZ^{m,\varepsilon}_s+\big[B^m(X^{m,\varepsilon}_s)-B^m(\bar{X}^m_{s})\big]\right.\nonumber\\
&&\quad\quad\quad\quad\quad\quad\quad\quad+\left.\big[F(X^{m,\varepsilon}_s, Y^{m,\varepsilon}_s)-\bar F(\bar{X}^m_{s})\big]\right\rangle ds\nonumber\\
&&+\int^t_0 \langle \mathscr{L}^m_{1}(Y^{m,\vare}_{s})\Phi_m(X^{m,\vare}_{s},Y^{m,\vare}_{s}), Z^{m,\varepsilon}_s \rangle ds\nonumber\\
&&+\frac{1}{\vare}\!\!\int^t_0 \!\!\langle \mathscr{L}^m_{2}(X_{s}^{m,\vare})\Phi_m(X^{m,\vare}_{s},Y^{m,\vare}_{s}), Z^{m,\varepsilon}_s\rangle ds\nonumber\\
&&+M^{m,\vare,1}_{t}+M^{m,\vare,2}_{t},\label{ID3}
\end{eqnarray}
where $\mathscr{L}^m_{1}(y)$ is the infinitesimal generator of the transition semigroup for the slow component $X^{\vare}$ with  fixed fast component $y$, i.e.,
\begin{eqnarray}
&&\mathscr{L}^m_{1}(y)\Phi_m(x,y):=D_x\Phi_m(x,y)\cdot\left[Ax+B^m(x)+F^m(x,y)\right]\nonumber\\
&&\quad\quad\quad\quad\quad\quad\quad\quad\quad\quad+\frac{1}{2}\sum^m_{k=1}\left[D_{xx}\Phi_m(x,y)\cdot\left( \sqrt{Q_1}e_k,\sqrt{Q_1}e_k\right)\right],\label{Lm_1}\\
&&M^{m,\vare,1}_{t}:=\int_0^t \langle D_x\Phi_m(X_{s}^{m,\vare},Y^{m,\vare}_{s})\cdot \sqrt{Q_1}d\bar{W}^{1,m}_s, Z^{m,\varepsilon}_s\rangle,\nonumber\label{M_1}\\
&&M^{m,\vare,2}_{t}:=\frac{1}{\sqrt{\vare}}\int_0^t  \langle D_y\Phi_m(X_{s}^{m,\vare},Y^{m,\vare}_{s})\cdot \sqrt{Q_2}d\bar{W}^{2,m}_s,\ Z^{m,\varepsilon}_s\rangle.\nonumber\label{M_2}
\end{eqnarray}
Then it follows
\begin{eqnarray}
&&\EE\left[\sup_{t\in [0,T]}|Z^{m,\varepsilon}_t|^p\right]\nonumber\\
\leq\!\!\!\!\!\!\!\!&&C_T (1+|x|^2)e^{2\delta|x|^2}\left[\EE\sup_{t\in[0,T]}\left|\int_{0}^{t}\langle\mathscr{L}^m_{2}(X_{s}^{m,\vare})\Phi_m(X_{s}^{m,\vare},Y^{m,\vare}_{s}), Z^{m,\varepsilon}_s\rangle ds\right|^p\right]^{1/2}\nonumber\\
\leq\!\!\!\!\!\!\!\!&&C_T(1+|x|) e^{\delta |x|^2}\vare^{p/2}\Bigg\{\left[\EE\sup_{t\in[0,T]}|\langle\Phi_m(X^{m,\vare}_{t},Y^{m,\vare}_{t}), Z^{m,\varepsilon}_t\rangle|^p\right]^{1/2}\Bigg\}\nonumber\\
&&+C_T(1+|x|) e^{\delta |x|^2}\vare^{p/2}\Bigg\{\Big[\EE \big|\int^T_0|\Phi_m(X^{m,\vare}_{s},Y^{m,\vare}_{s})|\cdot| AZ^{m,\varepsilon}_s+B^m(X^{m,\varepsilon}_s)-B^m(\bar{X}^m_{s})\nonumber\\
&&\quad\quad\quad\quad\quad\quad\quad\quad\quad\quad\quad\quad\quad+F^m(X^{m,\varepsilon}_s, Y^{m,\varepsilon}_s)-\bar F^m(\bar{X}^m_{s})|ds\big|^p\Big]^{1/2}\Bigg\}\nonumber\\
&&+C_T(1+|x|) e^{\delta |x|^2}\vare^{p/2}\Bigg\{\left[\EE\left|\int^T_0|Z^{m,\varepsilon}_s| |\mathscr{L}^m_{1}(Y^{m,\vare}_{s})\Phi_m(X_{s}^{m,\vare},Y^{m,\vare}_{s})|ds\right|^{p}\right]^{1/2}\Bigg\}\nonumber\\
&&+C_T(1+|x|) e^{\delta |x|^2}\vare^{p/2}\Bigg\{\left[\EE\left(\sup_{t\in[0,T]}|M^{m,\vare,1}_{t}|^p\right)\right]^{1/2}+\left[\EE\left(\sup_{t\in[0,T]}|M^{m,\vare,2}_{t}|^p\right)\right]^{1/2}\Bigg\}\nonumber\\
:=\!\!\!\!\!\!\!\!&&\sum^4_{k=1}\Gamma^{m,\vare}_k(T).\label{HF5.4}
\end{eqnarray}

For the term $\Gamma^{m,\vare}_1(T)$. By  \eref{AXvare}, \eref{A.1.0}, \eref{ASupYvare} and \eref{E0}, we have
\begin{eqnarray}
\Gamma^{m,\vare}_1(T)\leq\!\!\!\!\!\!\!\!&&C_T(1+|x|) e^{\delta |x|^2}\vare^{p/2}\Bigg[\EE\sup_{t\in[0,T]}|\Phi_m(X^{m,\vare}_{t},Y^{m,\vare}_{t})|^{2p}| Z^{m,\varepsilon}_t|^p\Bigg]^{1/4}\left\{\EE\left[\sup_{t\in[0,T]}| Z^{m,\varepsilon}_t|^p\right]\right\}^{1/4}\nonumber\\
\leq\!\!\!\!\!\!\!\!&& C_T(1+|x|^{\frac{4}{3}}) e^{\frac{4}{3}\delta |x|^2}\vare^{\frac{2p}{3}}\Bigg[\EE\left(\sup_{t\in[0,T]}|\Phi_m(X^{m,\vare}_{t},Y^{m,\vare}_{t})|^{4p}\right)\Bigg]^{1/6} \Bigg\{\EE\left[\sup_{t\in[0,T]}|Z^{m,\varepsilon}_t|^{2p}\right]\Bigg\}^{1/6}\nonumber\\
&&+\frac{1}{4}\EE\left[\sup_{t\in[0,T]}| Z^{m,\varepsilon}_t|^p\right]\nonumber\\
\leq\!\!\!\!\!\!\!\!&& C_T(1+|x|^{\frac{4}{3}}) e^{\frac{4}{3}\delta |x|^2}\vare^{\frac{2p}{3}}\Bigg[\EE\left(1+\sup_{t\in[0,T]}|Y^{m,\vare}_{t}|^{4p}\right)\Bigg]^{1/6} \nonumber\\
&&\cdot\Bigg\{\EE\left[\sup_{t\in[0,T]}\left(|X^{m,\varepsilon}_t|^{2p}+|\bar{X}^{m}_t|^{2p}\right)\right]\Bigg\}^{1/6}+\frac{1}{4}\EE\left[\sup_{t\in[0,T]}| Z^{m,\varepsilon}_t|^{p}\right]\nonumber\\
\leq\!\!\!\!\!\!\!\!&&C_T(1+|x|^{p}+|y|^{p}) e^{\frac{4\delta |x|^2}{3}}\vare^{p/2}+\frac{1}{4}\EE\left[\sup_{t\in[0,T]}| Z^{m,\varepsilon}_t|^p\right].\label{Gamma1}
\end{eqnarray}

For the term $\Gamma^{m,\vare}_2(T)$. By Minkowski's inequality, \eref{E0}, for any $\theta\in (0,1]$, $\gamma\in (1/2,1)$, there exists $d>1$ such that
\begin{eqnarray}
\Gamma^{m,\vare}_2(T)\leq\!\!\!\!\!\!\!\!&&C_{p,T}(1+|x|) e^{\delta |x|^2}\vare^{p/2}\Bigg\{\int^T_0\left[\EE|\Phi_m(X^{m,\vare}_{s},Y^{m,\vare}_{s})|^{2p}\right]^{\frac{1}{2p}}\cdot\Big[\EE\big(\|Z^{m,\varepsilon}_s\|^{2p}_2+|B^m(X^{m,\varepsilon}_s)|^{2p}\nonumber\\
&&+|B^m(\bar{X}^m_{s})|^{2p}+|F^m(X^{m,\varepsilon}_s, Y^{m,\varepsilon}_s)|^{2p}+|\bar F^m(\bar{X}^m_{s})|^{2p}\big)\Big]^{1/2p} ds\Bigg\}^{p/2}\nonumber\\
\leq\!\!\!\!\!\!\!\!&&C_{p,T}(1+|x|) e^{\delta |x|^2}\vare^{p/2}\Bigg\{\int^T_0\left[\EE(1+|Y^{m,\vare}_{s}|^{2p})\right]^{\frac{1}{2p}}\cdot\Big[\EE\big(1+\|X^{m,\varepsilon}_s\|^{2p}_2+\|\bar{X}^m_{s}\|^{2p}_2\nonumber\\
&&+\|X_{s}^{m,\vare}\|^{2p}_{1-\gamma}\|X_{s}^{m,\vare}\|^{2p}_{1/2+\gamma}+\|\bar{X}_{s}^{m}\|^{2p}_{1-\gamma}\|\bar{X}_{s}^{m}\|^{2p}_{1/2+\gamma}+|\bar{X}^m_{s}|^{2p}\big)\Big]^{\frac{1}{2p}} ds\Bigg\}^{p/2} \nonumber\\
\leq\!\!\!\!\!\!\!\!&&C_{p,T}(1+|x|^d+\|x\|^p_{\theta}+|y|^{p}) e^{\delta |x|^2}\vare^{p/2}\Bigg\{\int^T_0 \left(s^{-1+\frac{\theta}{2}}+s^{\frac{\alpha-\beta}{2}} \right)ds\Bigg\}^{p/2}\nonumber\\
\leq\!\!\!\!\!\!\!\!&&C_{p,T}(1+|x|^d+\|x\|^p_{\theta}+|y|^p)e^{2\delta |x|^2}\vare^{p/2}.\label{Gamma2}
\end{eqnarray}

For the term $\Gamma^{m,\vare}_3(T)$. Using \eref{E2}, \eref{E3} and similar as we did in estimating $\Gamma^{m,\vare}_2(T)$,  it is easy to prove that
\begin{eqnarray}
\Gamma^{m,\vare}_{3}(T)\leq\!\!\!\!\!\!\!\!&&C_T(1+|x|) e^{\delta |x|^2}\vare^{p/2}\Bigg\{\int^T_0\left[\EE|Z^{m,\varepsilon}_s|^{2p}\right]^{\frac{1}{2p}}\left[\EE|\mathscr{L}^m_{1}(Y^{m,\vare}_{s})\Phi_m(X_{s}^{m,\vare},Y^{m,\vare}_{s})|^{2p}\right]^{\frac{1}{2p}}ds\Bigg\}^{p/2}\nonumber\\
\leq\!\!\!\!\!\!\!\!&&C_T(1+|x|^p) e^{\delta |x|^2}\vare^{p/2}\Bigg\{\int^T_0 \Bigg[\EE\Big|D_x\Phi_m(X_{s}^{m,\vare},Y^{m,\vare}_{s})\cdot \big[AX_{s}^{m,\vare}+B(X_{s}^{m,\vare}) \nonumber\\
&&+F^m(X_{s}^{m,\vare},Y^{m,\vare}_{s})\big]+\frac{1}{2}\sum_{k=1}^m \left[D_{xx}\Phi_m(X_{s}^{m,\vare},Y^{m,\vare}_{s})\cdot(\sqrt{Q_1}e_k,\sqrt{Q_1}e_k)\right]\Big|^{2p}ds\Bigg]^{\frac{1}{2p}}ds\Bigg\}^{p/2}\nonumber\\
\leq\!\!\!\!\!\!\!\!&&C_{p,T}(1+|x|^d+\|x\|^p_{\theta}+|y|^p)e^{\delta |x|^2}\vare^{p/2}. \label{Gamma3}
\end{eqnarray}

By  Burkholder-Davis-Gundy's inequality, \eref{E1} and \eref{E2}, we have
\begin{eqnarray}
\Gamma^{m,\vare}_4(T)
\leq\!\!\!\!\!\!\!\!&&C_T (1+|x|)e^{\delta|x|^2}\vare^{p/2}\left\{\EE\left[\int^T_0|Z^{m,\varepsilon}_s|^2\| D_x\Phi_m(X_{s}^{m,\vare},Y^{m,\vare}_{s})\|^2\text{Tr}Q_1ds\right]^{p/2}\right\}^{1/2}\nonumber\\
&&+C_T (1+|x|)e^{\delta|x|^2}\vare^{p/4}\left\{\EE\left[\int^T_0|Z^{m,\varepsilon}_s|^2\| D_y\Phi_m(X_{s}^{m,\vare},Y^{m,\vare}_{s})\|^2\text{Tr}Q_2ds\right]^{p/2}\right\}^{1/2}\nonumber\\
\leq\!\!\!\!\!\!\!\!&&C_T (1+|x|)e^{\delta|x|^2}\vare^{p/2}\left\{\int^T_0\left(\EE|Z^{m,\varepsilon}_s|^{2p}\right)^{1/2}\left(\EE\| D_x\Phi_m(X_{s}^{m,\vare},Y^{m,\vare}_{s})\|^{2p}\right)^{1/2}ds\right\}^{1/2}\nonumber\\
&&+\frac{1}{4}\EE\left[\sup_{t\in [0,T]}|Z^{m,\varepsilon}_t|^{p}\right]+C_T (1+|x|^2)e^{2\delta|x|^2}\vare^{p/2}\nonumber\\
\leq\!\!\!\!\!\!\!\!&& C_{p,T}(1+|x|^{p}+|y|^{p})e^{4\delta|x|^2}\vare^{p/2}+\frac{1}{4}\EE\left[\sup_{t\in [0,T]}|Z^{m,\varepsilon}_t|^p\right].\label{Gamma4}
\end{eqnarray}

Hence, by \eref{HF5.4}-\eref{Gamma4}, we final have for any small enough $T>0$,
\begin{eqnarray}
\mathbb{E}\left(\sup_{t\in[0,T]}|Z^{m,\varepsilon}_t|^p\right)\leq C_{p,T}(1+|x|^d+\|x\|^p_{\theta}+|y|^p)e^{2\delta |x|^2}\vare.\label{SmallT}
\end{eqnarray}

Note that $(X^{m,\varepsilon}(x_1), Y^{m,\varepsilon}(y), \bar{X}^{m}(x_2)$ is a time homogeneous Markov process, which start from the initial value $(x_1,y,x_2)$. Assume that \eref{SmallT} holds for a small $T>0$. The remaining is sufficient to show \eref{SmallT} holds for $2T$.

In face, by a minor revision in the discussion above, it is easy to see that for any $p\geq 2$, there exist $d>1, C_{p,T}$ and small enough $\delta>0$ such that
\begin{eqnarray}
&&\EE\left[\sup_{t\in [0,T]}|X^{m,\varepsilon}(x_1)-\bar{X}^{m}_t(x_2)|^p\right]\nonumber\\
\leq\!\!\!\!\!\!\!\!&&C_{p,T}(1+|x_1|^d+|x_2|^d+\|x_1\|^{p}_{\theta}+\|x_2\|^{p}_{\theta}+|y|^{p}) e^{\delta (|x_1|^2+|x_2|^2)}\vare^{p/2}\nonumber\\
&&+C_T|x_1-x_2|^p(1+|x_1|^2+|x_2|^2)e^{\delta(|x_1|^2+|x_2|^2)}.\label{DI}
\end{eqnarray}
By Markov property, \eref{DI} and a prior estimates of $(X^{m,\vare}, Y^{m,\vare})$ and $\bar{X}^m$, we final get
\begin{eqnarray*}
\EE\left[\sup_{t\in [T,2T]}|X^{m,\varepsilon}_t-\bar{X}^{m}_t|^p\right]\leq\!\!\!\!\!\!\!\!&&\EE\left\{\EE\left[\sup_{t\in [T,2T]}|X^{m,\varepsilon}-\bar{X}^{m}_t|^p|\mathscr{F}^m_T\right]\right\}\\
\leq\!\!\!\!\!\!\!\!&&\EE\left\{\EE\left[\sup_{t\in [0,T]}|X^{m,\varepsilon}(x_1)-\bar{X}^{m}_t(x_2)|^p\right]\Bigg|_{\{x_1=X^{m,\varepsilon}_T, y=Y^{m,\varepsilon}_T, x_2=\bar{X}^{m}_T\}}\right\}\\
\leq\!\!\!\!\!\!\!\!&&C_{p,T}(1+|x|^d+\|x\|^p_{\theta}+|y|^p)e^{2\delta |x|^2}\vare^{p/2}\\
&&+\left[\mathbb{E}|X^{m,\varepsilon}_T-\bar{X}^{m}_T|^{2p}\right]^{1/2}\left[\EE(1+|X^{m,\varepsilon}_T|^4+|\bar{X}^{m}_T|^4)e^{2\delta(|X^{m,\varepsilon}_T|^2+|\bar{X}^{m}_T|^2)})\right]^{1/2}\\
\leq\!\!\!\!\!\!\!\!&&C_{p,T}(1+|x|^d+\|x\|^p_{\theta}+|y|^p)e^{2\delta |x|^2}\vare^{p/2},
\end{eqnarray*}
where $\mathscr{F}^m_T=\sigma\{(\bar{W}^{1,m}_t,\bar{W}^{2,m}_t),t\leq T\}$. The proof is complete.

\section{The proof of weak convergence}
 In this section, we shall prove Theorem \ref{main result 2}. To do this, we will give the idea of the proof in subsection 5.1 firstly. Then we study the regularity of solution of the Kolmogorov equation and Poisson equation in subsection 5.2. Finally, we give the detailed proofs of Theorem \ref{main result 2}.

\subsection{Idea of the proof of weak convergence} In order to show the key idea to readers for convenience, we here give a direct-viewing about how to use these techniques.

\vspace{0.1cm}
$\textbf{Step 1 (Finite dimensional approximation)}$
For a test function $\phi \in C_{b}^{3}(H)$, we have for any $t\geq 0$,
\begin{eqnarray} \label{ephix}
\left|\mathbb{E}\phi\left(X^{\vare}_t\right)
-\EE\phi(\bar{X}_t)\right|
\leq\!\!\!\!\!\!\!\!&&
\left|\mathbb{E}\phi\left(X^{\vare}_t\right)
-\mathbb{E}\phi\left(X^{m,\vare}_t\right)\right|+\left|\EE\phi (\bar{X}^m_t)-\EE\phi(\bar{X}_t)\right| \nonumber\\
\!\!\!\!\!\!\!\!&&+\left|\mathbb{E}\phi\left(X^{m,\vare}_t\right)-\EE\phi (\bar{X}^m_t)\right|.
\end{eqnarray}
By \eref{FA1} and \eref{FA2}, it is easy to see that for any $\vare>0$,
\begin{eqnarray*}
&&\lim_{m\rightarrow\infty}\sup_{t\in [0, T]}\left|\mathbb{E}\phi\left(X^{\vare}_t\right)
-\mathbb{E}\phi\left(X^{m,\vare}_t\right)\right|=0,\\
&&\lim_{m\rightarrow\infty}\sup_{t\in [0, T]}\left|\EE\phi (\bar{X}^m_t)-\EE\phi(\bar{X}_t)\right|=0.
\end{eqnarray*}
Then it is sufficient to show that there exists a  constant $C>0$ independent of $m$ such that
\begin{eqnarray}
\sup_{t\in [0, T]}\left|\mathbb{E}\phi\left(X^{m,\vare}_t\right)-\EE\phi (\bar{X}^m_t)\right|\leq C\vare. \label{MFE}
\end{eqnarray}

\vspace{0.1cm}
$\textbf{Step 2 (Kolmogorov equation)}$
In order to prove \eref{MFE}, we introduce the following Kolmogorov equation:
\begin{equation}\left\{\begin{array}{l}\label{KE}
\displaystyle
\partial_t u_m(t,x)=\bar{\mathscr{L}}^m_1 u_m(t,x),\quad t\in[0, T], \\
u_m(0, x)=\phi(x),\quad x\in H_m,
\end{array}\right.
\end{equation}
where $\phi\in C^{3}_b(H)$ and $\bar{\mathscr{L}}^m_1$ is the infinitesimal generator of the transition semigroup of the finite dimensional averaged equation \eref{Ga 1.3}, which is given by
\begin{eqnarray*}
\bar{\mathscr{L}}^m_{1}u_m(t,x):=\!\!\!\!\!\!\!\!&&D_x u_m(t,x)\cdot \left[Ax+B^m(x)+\bar{F}^m(x)\right]\\
&&+\frac{1}{2}\sum^m_{k=1}\left[D_{xx}u_m(t,x)\cdot(\sqrt{Q_1}e_k, \sqrt{Q_1}e_k) \right],\quad x\in H_m, t\in[0,T].
\end{eqnarray*}

For any fixed $t>0$, denote $\tilde{u}^t_m(s,x):=u_m(t-s,x)$, $s\in [0,t]$. If the solution $u_m(t,x)$ of Kolmogorov equation \eref{KE} is enough regularity (see Lemma \ref{Lemma3.6} below), applying It\^{o}'s formula, we get
\begin{eqnarray}
\tilde{u}^t_m(t, X^{m,\vare}_t)=\!\!\!\!\!\!\!\!&&\tilde{u}^t_m(0,x)+\int^t_0 \partial_s \tilde{u}^t_m(s, X^{m,\vare}_s )ds+\int^t_0 \mathscr{L}^m_{1}(Y^{m,\vare}_s)\tilde{u}^t_m(s, X^{m,\vare}_s)ds+\tilde{M}^m_t,\label{F5.4}
\end{eqnarray}
where
\begin{eqnarray*}
&&\mathscr{L}^m_{1}(y)\tilde{u}^t_m(s, x):=D_x\tilde{u}^t_m(s, x)\cdot\left[Ax+B^m(x)+F^m(x,y)\right]\\
&&\quad\quad\quad\quad\quad\quad\quad\quad\quad\quad+\frac{1}{2}\sum^m_{k=1}\left[D_{xx}\tilde{u}^t_m(s, x)\cdot\left( \sqrt{Q_1}e_k,\sqrt{Q_1}e_k\right)\right],\\
&&\tilde{M}^m_t:=\int^t_0\langle D_x\tilde{u}^t_m(s,X_s^{m,\varepsilon}),\sqrt{Q_1}d\bar{W}_s^{1,m}\rangle.
\end{eqnarray*}

Note that \eref{KE} admits a unique solution $u_m(t, x)=\EE\phi(\bar{X}^m_t(x))$, thus
\begin{eqnarray*}
&&\tilde{u}^t_m(t, X^{m,\vare}_t)=\phi(X^{m,\vare}_t),\quad \tilde{u}^t_m(0, x)=\EE\phi(\bar{X}^m_t(x)), \\
&&\partial_s \tilde{u}^t_m(s, X^{m,\vare}_s )=-\bar{\mathscr{L}}^m_1 \tilde{u}^t_m(s, X^{m,\vare}_s),
\end{eqnarray*}
which combine with \eref{F5.4}, we get
\begin{eqnarray}
\left|\EE\phi(X^{m,\vare}_{t})-\EE\phi(\bar{X}^m_{t})\right|=\!\!\!\!\!\!\!\!&&\left|\EE\int^t_0 \partial_s \tilde{u}^t_m(s, X^{m,\vare}_s )ds+\EE\int^t_0\mathscr{L}^m_{1}(Y^{m,\vare}_s)\tilde{u}^t_m(s, X^{m,\vare}_s)ds\right|\nonumber\\
=\!\!\!\!\!\!\!\!&&\left|\EE\int^t_0 -\bar{\mathscr{L}}^m_1 \tilde{u}^t_m(s, X^{m,\vare}_s )ds+\EE\int^t_0 \mathscr{L}^m_{1}(Y^{m,\vare}_s)\tilde{u}^t_m(s, X^{m,\vare}_s)ds\right|\nonumber\\
=\!\!\!\!\!\!\!\!&&\left|\EE\int^t_0 \langle F^m(X^{m,\vare}_s,Y^{m,\vare}_s)-\bar{F}^m(X^{m,\vare}_s), D_x \tilde{u}^t_m(s, X^{m,\vare}_s )\rangle ds \right|.\label{F5.11}
\end{eqnarray}

\vspace{0.2cm}
$\textbf{Step 3 (Poisson equation)}$ In order to estimate \eref{F5.11}, we consider the following Poisson equation in $H_m$:
\begin{eqnarray}
-\mathscr{L}^m_{2}(x)\Phi^t_m(s,x,y)=\langle F^m(x,y)-\bar{F}^m(x), D_x \tilde{u}^t_m(s, x)\rangle ,\quad s\in [0,t], x,y\in H_m,\label{WPE}
\end{eqnarray}
where $\mathscr{L}^m_{2}(x)$ is defined as follows:
\begin{eqnarray*}
\mathscr{L}^m_{2}(x)\Phi^t_m(s,x,y):=\!\!\!\!\!\!\!\!&&D_y\Phi^t_m(s,x,y)\cdot [Ay+G^m(x,y)]\nonumber\\
&&+\frac{1}{2}\sum^{m}_{k=1}\left[D_{yy}\Phi^t_m(s,x,y)\cdot(\sqrt{Q_2} e_k,  \sqrt{Q_2} e_k)\right].
\end{eqnarray*}

If $\Phi^t_m$ is a sufficiently regular solution of Poisson equation \eref{WPE} (see Proposition \ref{P3.6} below ), applying It\^o's formula and taking expectation, we get for any $t\in [0,T]$,
\begin{eqnarray*}
\EE\Phi^t_m(t, X_{t}^{m,\vare},Y^{m,\vare}_{t})=\!\!\!\!\!\!\!\!&& \Phi^t_m(0, x,y)+\EE\int^t_0 \partial_s \Phi^t_m(s, X_{s}^{m,\vare},Y^{m,\vare}_{s})ds\\
&&\!\!\!\!\!\!\!\!\!\!\!\!\!\!\!\!\!\!\!\!\!\!\!\!\!\!+\EE\int^t_0\mathscr{L}^m_{1}(Y^{m,\vare}_{s})\Phi^t_m(s, X_{s}^{m,\vare},Y^{m,\vare}_{s})ds+\frac{1}{\vare}\EE\int^t_0 \mathscr{L}^m_{2}(X_{s}^{m,\vare})\Phi^t_m(s, X_{s}^{m,\vare},Y^{m,\vare}_{s})ds,
\end{eqnarray*}
which implies
\begin{eqnarray}
&&-\EE\int^t_0 \mathscr{L}^m_{2}(X_{s}^{m,\vare})\Phi^t_m(s, X_{s}^{m,\vare},Y^{m,\vare}_{s})ds\nonumber\\
=\!\!\!\!\!\!\!\!&&\vare\big[\Phi^t_m(0, x,y)-\EE\Phi^t_m(t, X_{t}^{m,\vare},Y^{m,\vare}_{t})+\EE\int^t_0 \partial_s \Phi^t_m(s, X_{s}^{m,\vare},Y^{m,\vare}_{s})ds\nonumber\\
&&\quad\quad+\EE\int^t_0\mathscr{L}^m_{1}(Y^{m,\vare}_{s})\Phi^t_m(s, X_{s}^{m,\vare},Y^{m,\vare}_{s})ds\big].\label{F3.39}
\end{eqnarray}

Combining  \eref{F5.11}, \eref{WPE} and \eref{F3.39}, we get
\begin{eqnarray*}
&&\sup_{t\in[0,T]}\left|\EE\phi(X^{m,\vare}_{t})-\EE\phi(\bar{X}^m_{t})\right|=\sup_{t\in[0,T]}\left|\EE\int^t_0 \mathscr{L}^m_{2}(X_{s}^{m,\vare})\Phi^t_m(s, X_{s}^{m,\vare},Y^{m,\vare}_{s})ds\right|\\
\leq\!\!\!\!\!\!\!\!&&\vare\Bigg[\sup_{t\in [0,T]}|\Phi^t_m(0, x,y)|+\sup_{t\in[0,T]}\left|\EE\Phi^t_m(t, X_{t}^{m,\vare},Y^{m,\vare}_{t})\right|+\sup_{t\in [0,T]}\EE\int^t_0 \left|\partial_s \Phi^t_m(s, X_{s}^{m,\vare},Y^{m,\vare}_{s})\right|ds\nonumber\\
&&+\sup_{t\in [0,T]}\EE\int^t_0\left|\mathscr{L}^m_{1}(Y^{m,\vare}_{s})\Phi^t_m(s, X_{s}^{m,\vare},Y^{m,\vare}_{s})\right|ds\Bigg].
\end{eqnarray*}

Hence, the remaining work is to estimate all the terms above, which highly depends on the regularity of the solution $\Phi^t_m(x,y)$ of the Poisson equation \eref{WPE}.

\subsection{The regularity of solution of Kolmogorov equation and Poisson equation}

The directional derivative of $\bar{X}^m_t$ with respect to $x$ in the direction $h$ is denoted by $\eta^{h,m}_t(x)$, which satisfies the following equation
\begin{equation}\left\{\begin{array}{l} \label{Equa derivative}
\displaystyle
\frac{d\eta^{h,m}_t(x)}{dt}
=A\eta^{h,m}_t(x)+D\bar{F}^m(\bar{X}^m_t)\cdot\eta^{h,m}_t(x)
+D_{\xi}\left[\bar{X}^m_t\eta^{h,m}_t(x)\right]\\
\eta^{h,m}_0(x)=h\in H_m.
\end{array}\right.
\end{equation}

\begin{lemma} \label{eta}
The following two estimates hold.

(1) For any $T>0$, $h\in H_m$, there exist constants $C, C_T>0$ such that
\begin{align} \label{LemDeriEtah 01}
| \eta^{h,m}_t(x) |^2+\int^t_0\|\eta^{h,m}_s(x)\|_{1}^{2}ds\leq C_Te^{C \int^t_0\|\bar{X}^m_s\|^{4/3}_1ds}|h|^{2},\quad t\in [0,T].
\end{align}
Moreover, for any $p\geq 1$ and $T>0$, there exists small enough $\delta>0$ such that
\begin{eqnarray}
\sup_{t\in [0,T]}\left(\EE| \eta^{h,m}_t(x) |^{2p}\right)^{1/p}+\EE\int^T_0\|\eta^{h,m}_s(x)\|_{1}^{2}ds\leq C_T e^{\delta|x|^2}|h|^{2}.\label{ELemDeriEtah 01}
\end{eqnarray}

(2) For any $T>0$, $p>1$, $\gamma\in(1,3/2)$ and $h\in H^{\theta}$ with $\theta\in [0,\gamma]$, there exist a constant $C_{T}>0$ such that for any $t\in (0,T]$ and small enough $\delta>0$,
\begin{align} \label{LemDeriEtah 02}
\left[\EE\|\eta^{h,m}_t(x)\|^p_{\gamma}\right]^{1/p}
\leq C_T e^{\delta |x|^2}t^{-\frac{\gamma-\theta}{2}}\|h\|_{\theta}.
\end{align}
\end{lemma}

\begin{proof}
The proof \eref{LemDeriEtah 01} follows almost the same steps in \cite[Proposition 3.2]{DaD}. Note that by Young's inequality, it follows
$$
\|x\|^{4/3}_1\leq C+\delta\|x\|^2_1,
$$
where $\delta>0$ can be chosen small enough. Then \eref{ELemDeriEtah 01} is a consequence of \eref{LemDeriEtah 01} and \eref{A.1.1}. Hence the proofs of \eref{LemDeriEtah 01} and \eref{ELemDeriEtah 01} are omitted here. We only prove \eref{LemDeriEtah 02}.

Note that
\begin{align}\label{mild eta}
\eta^{h,m}_t(x)= e^{tA}h+\int_{0}^{t}e^{(t-s)A}D\bar{F}^m(\bar{X}^m_s)\cdot\eta^{h,m}_s(x)ds
+\int_{0}^{t}e^{(t-s)A}D_{\xi}\left[\bar{X}^m_s\eta^{h,m}_s(x)\right]ds.
\end{align}
Then for any $t\in (0,T]$, $\gamma\in(1,\frac{3}{2})$, $\delta'\in (1/2,2-\gamma)$ and $h\in H^{\theta}$ with $\theta\in [0,\gamma]$, we have
\begin{eqnarray}
\|\eta^{h,m}_t(x)\|_{\gamma}\leq\!\!\!\!\!\!\!\!&&
Ct^{-\frac{\gamma-\theta}{2}}\|h\|_{\theta}
+C\int_{0}^{t}(t-s)^{-\frac{\gamma}{2}}|\eta^{h,m}_s(x)|ds \nonumber\\
\!\!\!\!\!\!\!\!&&
+C\int_{0}^{t}(t-s)^{-\frac{\gamma+\delta'}{2}}\left\|D_{\xi}\left[\bar{X}^m_s\eta^{h,m}_s(x)\right]\right\|_{-\delta'}ds \nonumber\\
\leq\!\!\!\!\!\!\!\!&&
Ct^{-\frac{\gamma-\theta}{2}}\|h\|_{\theta}
+C\int_{0}^{t}(t-s)^{-\frac{\gamma}{2}}|\eta^{h,m}_s(x)|ds  \nonumber\\
\!\!\!\!\!\!\!\!&&
+C\int_{0}^{t}(t-s)^{-\frac{\gamma+\delta'}{2}}\left\|B(\bar{X}^m_s,\eta^{h,m}_s(x))+B(\eta^{h,m}_s(x),\bar{X}^m_s)\right\|_{-\delta'}ds \nonumber\\
\leq\!\!\!\!\!\!\!\!&&
Ct^{-\frac{\gamma-\theta}{2}}\|h\|_{\theta}+C\int_{0}^{t}(t-s)^{-\frac{\gamma}{2}}|\eta^{h,m}_s(x)|ds \nonumber\\
\!\!\!\!\!\!\!\!&&
+C\int_{0}^{t}(t-s)^{-\frac{\gamma+\delta'}{2}}\left[\|\eta^{h,m}_s(x)\|_{1} |\bar{X}^m_s|+|\eta^{h,m}_s(x)| \|\bar{X}^m_s\|_1\right]ds \nonumber\\
\leq\!\!\!\!\!\!\!\!&&
Ct^{-\frac{\gamma-\theta}{2}}\|h\|_{\theta}+C\int_{0}^{t}(t-s)^{-\frac{\gamma}{2}}|\eta^{h,m}_s(x)|ds \nonumber\\
\!\!\!\!\!\!\!\!&&
+C\int_{0}^{t}(t-s)^{-\frac{\gamma+\delta'}{2}}\left[|\eta^{h,m}_s(x)|^{\frac{\gamma-1}{\gamma}} \|\eta^{h,m}_s(x)\|^{\frac{1}{\gamma}}_{\gamma} |\bar{X}^m_s|+|\eta^{h,m}_s(x)| \|\bar{X}^m_s\|_1\right]ds. \nonumber
\end{eqnarray}

By Minkowski's inequality, \eref{ELemDeriEtah 01}, it follows for any $p\geq 1$,
\begin{eqnarray*}
\left[\EE\|\eta^{h,m}_t(x)\|^p_{\gamma}\right]^{1/p}
\leq\!\!\!\!\!\!\!\!&&
Ct^{-\frac{\gamma-\theta}{2}}\|h\|_{\theta}
+C\int_{0}^{t}(t-s)^{-\frac{\gamma}{2}}\left[\EE|\eta^{h,m}_s(x)|^p\right]^{1/p}ds\nonumber\\
\!\!\!\!\!\!\!\!&&
+C\int_{0}^{t}(t-s)^{-\frac{\gamma+\delta'}{2}}\left[\EE\left[|\eta^{h,m}_s(x)|^{\frac{p(\gamma-1)}{\gamma}} \|\eta^{h,m}_s(x)\|^{\frac{p}{\gamma}}_{\gamma} |\bar{X}^m_s|^p\right]\right]^{1/p}ds\nonumber\\
\!\!\!\!\!\!\!\!&& +C\int_{0}^{t}(t-s)^{-\frac{\gamma+\delta'}{2}}\left[\EE |\eta^{h,m}_s(x)|^p \|\bar{X}^m_s\|^p_1\right]^{1/p}ds \\
\leq\!\!\!\!\!\!\!\!&&
Ct^{-\frac{\gamma-\theta}{2}}\|h\|_{\theta}
+C\int_{0}^{t}(t-s)^{-\frac{\gamma}{2}}\left[\EE|\eta^{h,m}_s(x)|^p\right]^{1/p}ds  \nonumber\\
\!\!\!\!\!\!\!\!&&
+C\int_{0}^{t}(t-s)^{-\frac{\gamma+\delta'}{2}}\left[\EE\left[|\eta^{h,m}_s(x)|^{p} |\bar{X}^m_s|^{\frac{p\gamma}{\gamma-1}}+\|\eta^{h,m}_s(x)\|^{p}_{\gamma}\right]\right]^{1/p}ds\\
\!\!\!\!\!\!\!\!&& +C\int_{0}^{t}(t-s)^{-\frac{\gamma+\delta'}{2}}\left[\EE |\eta^{h,m}_s(x)|^{2p}\right]^{\frac{1}{2p}}\left[\EE \|\bar{X}^m_s\|^{2p}_1\right]^{\frac{1}{2p}}ds \\
\leq\!\!\!\!\!\!\!\!&&
C_Te^{\delta |x|^2}t^{-\frac{\gamma-\theta}{2}}\|h\|_{\theta}
+C_T\int_{0}^{t}(t-s)^{-\frac{\gamma}{2}}\left[\EE\|\eta^{h,m}_s(x)\|^p_{\gamma}\right]^{1/p}ds.
\end{eqnarray*}
where $\delta$ is a small enough positive constant.

Consequently, Lemma \ref{Gronwall 2} implies that \eqref{LemDeriEtah 02} holds. The proof is complete.
\end{proof}

\begin{remark}
By Lemma \ref{eta} and interpolation inequality, we deduce that for
any $T>0$, $t\in (0, T]$, $p\geq 1$, $\gamma_1\in (0, 1]$, $\gamma\in (1,3/2)$ and $\theta\in [0,1]$,
there exists a constant $C_{T}>0$ and small enough $\delta>0$ such that
\begin{align} \label{eta1}
\left[\EE\|\eta^{h,m}_t(x)\|^p_{\gamma_1}\right]^{1/p}\leq
C\left[\EE\left(|\eta^{h,m}_t(x)|^{\frac{(\gamma-\gamma_1)p}{\gamma}}
\|\eta^{h,m}_t(x)\|^{\frac{\gamma_1 p}{\gamma}}_{\gamma}\right)\right]^{1/p}
\leq C_{T}e^{\delta|x|^2}t^{-\frac{(\gamma-\theta) \gamma_1 }{2\gamma}}\|h\|_{\theta}.
\end{align}
\end{remark}

\begin{lemma}\label{LemmaA.5}
For any $\gamma\in (1,3/2)$, $\theta\in (0,\gamma-1]$, and $0<s<t\leq T$,
there exists a constant $C_{T}>0$ such that for small enough $\delta>0$,
\begin{align}
\left[\EE\|\eta^{h,m}_t(x)-\eta^{h,m}_s(x)\|^2_{1}\right]^{1/2}
\leq
C_T(t-s)^{\frac{\gamma-1}{2}}s^{-\frac{\gamma-\theta}{2}}e^{\delta |x|^{2}}\|h\|_{\theta} . \label{F5.14}
\end{align}
\end{lemma}
\begin{proof}
From \eqref{Equa derivative}, we write
\begin{eqnarray} \label{ContinI5}
\eta^{h,m}_t(x)-\eta^{h,m}_s(x)=\!\!\!\!\!\!\!\!&&(e^{(t-s)A}-I)\eta^{h,m}_s(x)+\int_{s}^{t}e^{(t-r)A}D\bar{F}^m(\bar{X}^m_r)\cdot\eta^{h,m}_r(x)dr \nonumber\\
\!\!\!\!\!\!\!\!&&+\int_{s}^{t}e^{(t-r)A}D_{\xi}\left[\bar{X}^m_r\eta^{h,m}_r(x)\right]dr.
\end{eqnarray}

Using \eqref{P4} and \eqref{LemDeriEtah 02}, we have
\begin{eqnarray}  \label{I51}
\left[\EE\|(e^{(t-s)A}-I)\eta^{h,m}_s(x)\|^2_{1}\right]^{1/2}\leq\!\!\!\!\!\!\!\!&& C(t-s)^{\frac{\gamma-1}{2}}\left[\EE\|\eta^{h,m}_s(x)\|^2_{\gamma}\right]^{1/2}\nonumber\\
\leq\!\!\!\!\!\!\!\!&&  C_T(t-s)^{\frac{\gamma-1}{2}}s^{-\frac{\gamma-\theta}{2}}e^{\delta |x|^2}\|h\|_{\theta}.
\end{eqnarray}

By \eqref{ELemDeriEtah 01}, we get
\begin{eqnarray}   \label{I52}
\left[\EE\left\|\int_{s}^{t}e^{(t-r)A}D\bar{F}^m(\bar{X}^m_r)\cdot\eta^{h,m}_r(x)dr\right \|^2_{1}\right]^{1/2}
\leq\!\!\!\!\!\!\!\!&&\left[\EE\left|\int_{s}^{t}(t-r)^{-\frac{1}{2}}|D\bar{F}^m(\bar{X}^m_r)\cdot \eta^{h,m}_r(x)| dr\right|^2\right]^{1/2}\nonumber\\
\leq\!\!\!\!\!\!\!\!&&  C_T(t-s)^{1/2}e^{\delta|x|^{2}}|h|.
\end{eqnarray}

By Lemmas \ref{BarXgamma} and \ref{eta}, we obtain for any $\gamma\in (1,3/2)$,
\begin{eqnarray}  \label{I53}
&&\left[\EE\left\|\int_{s}^{t}e^{(t-r)A}D_{\xi}\left[\bar{X}^m_r\eta^{h,m}_r(x)\right]dr\right\|^2_{1}\right]^{1/2}\nonumber\\
\leq\!\!\!\!\!\!\!\!&&
\left[\EE\left[\int_{s}^{t}(t-r)^{-\frac{3-\gamma}{2}}
\left\|B^m(\bar{X}^m_r,\eta^{h,m}_r(x))+B^m(\eta^{h,m}_r(x),\bar{X}^m_r)\right\|_{-(2-\gamma)}dr\right]^{2}\right]^{1/2}  \nonumber\\
\leq\!\!\!\!\!\!\!\!&&C\int_{s}^{t}(t-r)^{-\frac{3-\gamma}{2}}\left[\EE\left(|\bar{X}^m_r|^2\|\eta^{h,m}_r(x)\|^2_{1}+\|\bar{X}^m_r\|^2_{1}|\eta^{h,m}_r(x)|^2\right)\right]^{1/2}dr\nonumber\\
\leq\!\!\!\!\!\!\!\!&&
C_T(t-s)^{\frac{\gamma-1}{2}}s^{-\frac{1}{2}}e^{\delta|x|^{2}}|h| .
\end{eqnarray}

Hence by  \eqref{ContinI5}-\eqref{I53}, we easily obtain \eref{F5.14} holds. The proof is complete.
\end{proof}

\begin{lemma} \label{ESDET} For any $\theta\in (0,1]$ and $0<s<t\leq T$,  there exists a constant $C_{T}>0$ and small enough $\delta>0$ such that
\begin{align}
\EE\left\|\eta^{h,m}_t(x)\right\|_2\leq C_Te^{\delta|x|^2}(1+\|x\|_{\theta})\|h\|_{\theta}t^{-1+\theta/2}.  \nonumber
\end{align}
\end{lemma}
\begin{proof}
We first control the first term in \eqref{Equa derivative}.
Notice that
\begin{eqnarray}
\eta^{h,m}_t(x)
=\!\!\!\!\!\!\!\!&&e^{tA}h+\int_{0}^{t}e^{(t-s)A}D\bar{F}^m(\bar{X}^m_t)\cdot\eta^{h,m}_t(x)ds \nonumber\\
\!\!\!\!\!\!\!\!&&+\int_{0}^{t}e^{(t-s)A}\left[D\bar{F}^m(\bar{X}^m_s)\cdot\eta^{h,m}_s(x)-D\bar{F}^m(\bar{X}^m_t)\cdot\eta^{h,m}_t(x)\right]ds  \nonumber\\
\!\!\!\!\!\!\!\!&&+\int_{0}^{t}e^{(t-s)A}D_{\xi}\left[\bar{X}^m_t\eta^{h,m}_t(x)\right]ds \nonumber\\
\!\!\!\!\!\!\!\!&&+\int_{0}^{t}e^{(t-s)A}\left\{D_{\xi}\left[\bar{X}^m_s\eta^{h,m}_s(x)\right]-D_{\xi}\left[\bar{X}^m_t\eta^{h,m}_t(x)\right]\right\}ds \nonumber\\
=\!\!\!\!\!\!\!\!&&\sum^5_{k=1}J_{k}(t). \nonumber
\end{eqnarray}

For the term $J_{1}(t)$. By \eref{P3} we have
\begin{align} \label{AI1}
\|J_1(t)\|_2 \leq C t^{-1+\theta/2} \|h\|_{\theta}.
\end{align}

For the term $J_{2}(t)$. By \eref{ELemDeriEtah 01}, there exists a small enough $\delta>0$ such that
\begin{align} \label{AI2}
\EE\|J_{2}(t)\|_2= \EE\left |(e^{tA}-I)D\bar{F}^m(\bar{X}^m_t)\cdot\eta^{h,m}_t(x)\right|
\leq  C\EE\left|\eta^{h,m}_t(x)\right|
\leq  C_Te^{\delta|x|^{2}} |h|.
\end{align}

For the term $J_{3}(t)$. According to \eref{COXT} and \eref{F5.14}, we obtain
\begin{eqnarray} \label{AI3}
\EE\| J_{3}(t) \|_2
\leq\!\!\!\!\!\!\!\!&&
C\EE\int_{0}^{t}(t-s)^{-1}
\left\| D\bar{F}^m(\bar{X}^m_t)\cdot\eta^{h,m}_t(x)-D\bar{F}^m(\bar{X}^m_s)\cdot\eta^{h,m}_s(x)\right\|ds
\nonumber\\
\leq\!\!\!\!\!\!\!\!&&
C\EE\int_{0}^{t}(t-s)^{-1}\left\|[D\bar{F}^m(\bar{X}^m_t)-D\bar{F}^m(\bar{X}^m_s)]\cdot\eta^{h,m}_t(x)\right\|ds \nonumber\\
\!\!\!\!\!\!\!\!&&
+C\EE\int_{0}^{t}(t-s)^{-1}\left\|D\bar{F}^m(\bar{X}^m_s)\cdot
\left(\eta^{h,m}_t(x)-\eta^{h,m}_s(x)\right)\right\|ds \nonumber\\
\leq\!\!\!\!\!\!\!\!&&
C\int_{0}^{t}(t-s)^{-1}\left[\EE\|\bar{X}^m_t-\bar{X}^m_s\|^2_{1}\right]^{1/2}\left[\EE|\eta^{h,m}_t(x)|^2\right]^{1/2}ds\nonumber\\
&&+ C\int_{0}^{t}(t-s)^{-1} \left[\EE\left\|\eta^{h,m}_t(x)-\eta^{h,m}_s(x)\right\|^2_{1}\right]^{1/2}ds
\nonumber\\
\leq\!\!\!\!\!\!\!\!&&
C_Te^{\delta |x|^{2}}|h|\int_{0}^{t}(t-s)^{-1}(t-s)^{\frac{\gamma-1}{2}}s^{-\frac{\gamma}{2}}ds\nonumber\\
\leq\!\!\!\!\!\!\!\!&&C_Tt^{-\frac{1}{2}}e^{\delta |x|^{2}}|h|.
\end{eqnarray}

For the term $J_{4}(t)$. By Lemma \ref{Property B1}, \eqref{BarX1} and \eqref{eta1},  we deduce
\begin{eqnarray} \label{AI4}
\EE\|J_{4}(t)\|_2=\!\!\!\!\!\!\!\!&&
\EE\left\|(e^{tA}-I)D_{\xi}\left[\bar{X}^m_t\eta^{h,m}_t(x)\right]\right\| \nonumber\\
\leq\!\!\!\!\!\!\!\!&&
2 \EE\| B^m(\bar{X}^m_t,\eta^{h,m}_t(x))+B^m(\eta^{h,m}_t(x),\bar{X}^m_t)\|       \nonumber\\
\leq\!\!\!\!\!\!\!\!&&
C\left[\EE\left\|\bar{X}^m_t\right\|^2_{\frac{3}{2}-\gamma}\right]^{1/2}\left[\EE\left\|\eta^{h,m}_t(x)\right\|^2_{\gamma}\right]^{1/2}\nonumber\\
&&+C\left[\EE\left\|\eta^{h,m}_t(x)\right\|^2_{\frac{3}{2}-\gamma}\right]^{1/2}\left[\EE\left\|\bar{X}^m_t\right\|^2_{\gamma}\right]^{1/2}\nonumber\\
\leq  \!\!\!\!\!\!\!\!&&
Ct^{-1+\theta/2}e^{\delta |x|^{2}}\|h\|_{\theta}.
\end{eqnarray}

For the term $J_{5}(t)$.  it follows that for small enough $\delta'\in (0,1/2)$,
we get
\begin{eqnarray} \label{AI5}
\EE\|J_{5}(t)\|_2\leq\!\!\!\!\!\!\!\!&&
C\EE\Big|\int_{0}^{t}(-A)e^{(t-s)A}\Big[B^m\left(\bar{X}^m_s, \eta^{h,m}_t(x)-\eta^{h,m}_s(x)\right)+B^m\left(\eta^{h,m}_t(x)-\eta^{h,m}_s(x), \bar{X}^m_s\right)\nonumber\\
\!\!\!\!\!\!\!\!&&
+B^m\left(\bar{X}^m_t-\bar{X}^m_s, \eta^{h,m}_t(x)\right)+B^m\left(\eta^{h,m}_t(x), \bar{X}^m_t-\bar{X}^m_s\right)\Big]ds\Big| \nonumber\\
\leq\!\!\!\!\!\!\!\!&&
C\int_{0}^{t}(t-s)^{-1-\frac{\delta'}{2}}\EE\|B^m\left(\bar{X}^m_s, \eta^{h,m}_t(x)-\eta^{h,m}_s(x)\right)\|_{-\delta'}ds\nonumber\\
&&+C\int_{0}^{t}(t-s)^{-1-\frac{\delta'}{2}}\EE\|B^m\left(\eta^{h,m}_t(x)-\eta^{h,m}_s(x), \bar{X}^m_s\right)\|_{-\delta'}ds\nonumber\\
&&+C\int_{0}^{t}(t-s)^{-1}\EE|B\left(\bar{X}^m_t-\bar{X}^m_s, \eta^{h,m}_t(x)\right)|ds\nonumber\\
&&+C\int_{0}^{t}(t-s)^{-1}\EE|B\left(\eta^{h,m}_t(x), \bar{X}^m_t-\bar{X}^m_s\right)|ds\nonumber\\
:=\!\!\!\!\!\!\!\!&&\sum^4_{k=1}J_{5k}(t).
\end{eqnarray}
According to \eref{A.1.0}, \eref{BarX1} and \eref{COXT}, it follows
\begin{eqnarray} \label{J51}
J_{51}(t)\leq\!\!\!\!\!\!\!\!&& C\int_{0}^{t}(t-s)^{-1-\frac{\delta'}{2}}\EE\left(\|\bar{X}^m_s\|_{1/2-\delta'} \|\eta^{h,m}_t(x)-\eta^{h,m}_s(x)\|_{1}\right)ds\nonumber\\
\leq\!\!\!\!\!\!\!\!&& C\int_{0}^{t}(t-s)^{-1-\frac{\delta'}{2}}\left(\EE\|\bar{X}^m_s\|^2_{1/2-\delta'}\right)^{1/2}\left(\EE\|\eta^{h,m}_t(x)-\eta^{h,m}_s(x)\|^2_{1}\right)^{1/2}ds\nonumber\\
\leq\!\!\!\!\!\!\!\!&& C_T e^{\delta|x|^2}\|h\|_{\theta}\int_{0}^{t}(t-s)^{-1-\frac{\delta'}{2}}(t-s)^{\frac{\gamma-1}{2}}s^{-\frac{\gamma-\theta}{2}}s^{-\frac{1/2-\delta'}{2}}ds\nonumber\\
\leq\!\!\!\!\!\!\!\!&& C_T e^{\delta|x|^2}\|h\|_{\theta} t^{-1+\theta/2}.
\end{eqnarray}
By \eref{ELemDeriEtah 01}, \eref{BarX1} and \eref{F5.14}, there exists small enough $\delta>0$ such that for small enough $\delta'\in (0,1/2)$,
\begin{eqnarray} \label{J52}
J_{52}(t)\leq\!\!\!\!\!\!\!\!&& C\int_{0}^{t}(t-s)^{-1-\frac{\delta'}{2}}\EE\left(\|\eta^{h,m}_t(x)-\eta^{h,m}_s(x)\|_{1/2} \|\bar{X}^m_s\|_{1}\right)ds\nonumber\\
\leq\!\!\!\!\!\!\!\!&& C\int_{0}^{t}(t-s)^{-1-\frac{\delta'}{2}}\left(\EE|\eta^{h,m}_t(x)-\eta^{h,m}_s(x)|^2\right)^{1/4}\nonumber\\
&&\quad\quad\quad\quad\cdot \left(\EE\|\eta^{h,m}_t(x)-\eta^{h,m}_s(x)\|^2_{1}\right)^{1/4} \left(\EE\|\bar{X}^m_s\|^2_{1}\right)^{1/2}ds\nonumber\\
\leq\!\!\!\!\!\!\!\!&& C_Te^{\delta|x|^2}\|h\|_{\theta}\int_{0}^{t}(t-s)^{-1-\frac{\delta'}{2}}(t-s)^{\frac{\gamma-1}{4}}s^{-\frac{\gamma-\theta}{4}}s^{-\frac{1}{2}}ds\nonumber\\
\leq\!\!\!\!\!\!\!\!&& C_Te^{\delta|x|^2}\|h\|_{\theta}t^{-1+\theta/2}.
\end{eqnarray}
By \eref{BarX1} and  \eref{eta1}, we have
\begin{eqnarray} \label{J53}
J_{53}(t)+J_{54}(t)\leq\!\!\!\!\!\!\!\!&& C\int_{0}^{t}(t-s)^{-1}\EE\left(\|\bar{X}^m_t-\bar{X}^m_s\|_{1} \| \eta^{h,m}_t(x)\|_{1}\right)ds\nonumber\\
\leq\!\!\!\!\!\!\!\!&& C\int_{0}^{t}(t-s)^{-1}\left(\EE\|\bar{X}^m_t-\bar{X}^m_s\|^2_{1}\right)^{1/2} \left(\EE\| \eta^{h,m}_t(x)\|^2_{1}\right)^{1/2}ds\nonumber\\
\leq\!\!\!\!\!\!\!\!&& C_Te^{\delta|x|^2}(1+\|x\|_{\theta})|h|\int_{0}^{t}(t-s)^{-1+\frac{\gamma-1}{2}}s^{-\frac{\gamma-\theta}{2}}t^{-1/2}ds\nonumber\\
\leq\!\!\!\!\!\!\!\!&& C_Te^{\delta|x|^2}(1+\|x\|_{\theta})|h|t^{-1+\theta/2}.
\end{eqnarray}

Finally, by \eref{AI1}-\eref{J53}, we get the desired result. The proof is complete.
\end{proof}

Denote by $\zeta^{h,k,m}_t(x)$  the second derivative of $\bar{X}^m_t$
with respect to $x$ in the direction $(h,k)$. Then, $\zeta^{h,k,m}_t(x)$ satisfies the following equation:
 \begin{equation}\left\{\begin{array}{l}
\displaystyle
\frac{d\zeta^{h,k,m}_t(x)}{dt}=A\zeta^{h,k,m}_t(x)+\!\!\left[D\bar{F}^m(\bar{X}^m_t)\cdot \zeta^{h,k,m}_t(x)+D^2\bar{F}^m(\bar{X}^m_t)\cdot (\eta^{h,m}_t(x), \eta^{k,m}_t(x))\right]\!\!\nonumber\\
\quad\quad\quad\quad\quad\quad+D_{\xi}\left[\eta^{k,m}_t(x)\eta^{h,m}_t(x)\right]+D_{\xi}\left[\bar{X}^m_t(x)\zeta^{h,k,m}_t(x)\right],\nonumber\\
\zeta^{h,k,m}_0(x)=0.\nonumber
\end{array}\right.
\end{equation}
By \eref{D1Fm}, \eref{D2Fm}, \eref{eta1} and do some minor revisions in \cite[Propositions 3.3 and 3.4]{DaD}, we can easily prove that for any $T>0$, there exist positive constants $C, C_{T}$ such that
\begin{eqnarray}
|\zeta^{h,k,m}_t(x)|^2+\int^t_0 \|\zeta^{h,k,m}_s(x)\|^2_{1} ds\leq C_Te^{C\int^t_0 \|\bar{X}^m_s\|^{4/3}_{1} ds }|h|^2|k|^2,\quad t\in [0,T]. \label{Dxx barX}
\end{eqnarray}
Then by Lemma \ref{BarXgamma}, there exists small enough $\delta>0$ such that
\begin{eqnarray}
\sup_{t\in [0,T]}\EE|\zeta^{h,k,m}_t(x)|^2+\int^T_0\EE \|\zeta^{h,k,m}_s(x)\|^2_1 ds\leq C_{T} e^{\delta |x|^2}|h|^2|k|^2. \label{EDxx barX}
\end{eqnarray}

Denote by $\chi^{h,k,l}_t(x)$  the third derivative of $\bar{X}_t$
with respect to $x$ in the direction $(h,k,l)$. Then, $\chi^{h,k,l}_t(x)$ satisfies the following equation:
  \begin{equation}\left\{\begin{array}{l}
\displaystyle
\frac{d\chi^{h,k,l,m}_t(x)}{dt}=A\chi^{h,k,m}_t(x)+\!\!D\bar{F}^m(\bar{X}^m_t)\cdot \chi^{h,k,l,m}_t(x)+D^2\bar{F}^m(\bar{X}^m_t)\cdot (\zeta^{h,k,m}_t(x), \eta^{l,m}_t(x))\nonumber\\
\quad\quad\quad\quad\quad\quad   +D^2\bar{F}^m(\bar{X}^m_t)\cdot (\zeta^{h,l,m}_t(x),\eta^{k,m}_t(x))+D^2\bar{F}^m(\bar{X}^m_t)\cdot (\eta^{h,m}_t(x),\zeta^{k,l,m}_t(x))\nonumber\\
\quad\quad\quad\quad\quad\quad  +D^3\bar{F}^m(\bar{X}^m_t)\cdot (\eta^{h,m}_t(x), \eta^{k,m}_t(x), \eta^{l,m}_t(x))+D_{\xi}\left[\zeta^{k,l,m}_t(x)\eta^{h,m}_t(x)\right]\nonumber\\
\quad\quad\quad\quad\quad\quad +D_{\xi}\left[\eta^{k,m}_t(x)\zeta^{h,l,m}_t(x)\right]+D_{\xi}\left[\eta^{l,m}_t(x)\zeta^{h,k,m}_t(x)\right]+D_{\xi}\left[\bar{X}^m_t(x)\chi^{h,k,l,m}_t(x)\right]\nonumber\\
\chi^{h,k,l,m}_0(x)=0.\nonumber
\end{array}\right.
\end{equation}
By \eref{D1Fm}-\eref{D3Fm}, \eref{ELemDeriEtah 01}, \eref{eta1}, \eref{Dxx barX} and a straightforward computation,  we can prove that for any $T>0$, there exists $C_{T}$ and small enough $\delta>0$ such that
\begin{eqnarray}
\sup_{t\in [0,T]}\EE|\chi^{h,k,l,m}_t(x)|^2\leq C_{T}|h|^2|k|^2|l|^2 e^{\delta |x|^2}.\label{Dxxx barX}
\end{eqnarray}

By the preparation above, we present the regularity of the solution of equation \eref{KE}.
\begin{lemma} \label{Lemma3.6}
For any $h,k,l\in H_m$, $\theta\in(0,1]$, there exists small enough $\delta>0$ such that
\begin{eqnarray}
&&\sup_{m\geq 1}|\partial_t(D_x u_m(t,x))\cdot h|\leq C_T t^{-1+\theta/2}\|h\|_{\theta}(1+\|x\|_{\theta})e^{\delta|x|^2},\quad t\in(0,T],\label{UE2}\\
&&\sup_{m\geq 1}|D_{x} u_m(t,x)\cdot h|\leq C_{T}|h|e^{\delta|x|^2},\quad t\in[0,T],\label{UE3}\\
&&\sup_{m\geq 1}|D_{xx} u_m(t,x)\cdot (h,k)|\leq C_{T} |h||k|e^{\delta|x|^2},\quad t\in[0,T],\label{UE4}\\
&&\sup_{m\geq 1}|D_{xxx} u_m(t,x)\cdot (h,k,l)|\leq C_{T} |h||k||l|e^{\delta|x|^2},\quad  t\in[0,T].\label{UE5}
\end{eqnarray}
\end{lemma}
\begin{proof}
For any $h,k\in H_m$, note that
\begin{eqnarray*}
&&D_x u_m(t,x)\cdot h=\EE[D\phi(\bar{X}^m_t)\cdot \eta^{h,m}_t(x)],\\
&&D_{xx} u_m(t,x)\cdot (h,k)=\EE\left[D^2\phi(\bar{X}^m_t)\cdot (\eta^{h,m}_t(x),\eta^{k,m}_t(x))\right]\\
&&\quad\quad\quad\quad\quad\quad\quad\quad\quad+ \EE\left[D\phi(\bar{X}^m_t)\cdot  \zeta^{h,k,m}_t(x)\right],\\
&&D_{xxx} u_m(t,x)\cdot (h,k,l)=\EE\left[D^3\phi(\bar{X}^m_t)\cdot (\eta^{h,m}_t(x),\eta^{k,m}_t(x),\eta^{l,m}_t(x))\right]\\
&&\quad\quad\quad+\EE\left[D^2\phi(\bar{X}^m_t)\cdot (\zeta^{h,l,m}_t(x),\eta^{k,m}_t(x))+D^2\phi(\bar{X}^m_t)\cdot (\eta^{h,m}_t(x),\zeta^{k,l,m}_t(x))\right]\\
&&\quad\quad\quad+\EE\left[D^2\phi(\bar{X}^m_t)\cdot  (\zeta^{h,k,m}_t(x), \eta^{l,m}_t(x))+D\phi(\bar{X}^m_t)\cdot  \chi^{h,k,l,m}_t(x)\right].
\end{eqnarray*}
Since $\phi\in C^{3}_b(H)$, then \eref{UE3}-\eref{UE5} can easily obtained by \eref{ELemDeriEtah 01}, \eref{Dxx barX} and \eref{Dxxx barX}.
Now, we are going to prove \eref{UE2}.

By It\^{o}'s formula and taking expectation, we have
\begin{eqnarray*}
\EE[D\phi(\bar{X}^m_t)\cdot \eta^{h,m}_t(x)]=\!\!\!\!\!\!\!\!&&D\phi(x)\cdot h+\int^t_0 \EE\Big[ D^2\phi(\bar{X}^m_s)\cdot \left (\eta^{h,m}_s(x), A\bar{X}^m_s+B^m(\bar{X}^m_s)+\bar{F}^m(\bar{X}^m_s)\right)\Big]ds\\
&&\!\!\!\!\!\!\!\!\!\!\!\!\!\!\!\!+ \int^t_0\EE\Big[D\phi(\bar{X}^m_s)\cdot\left(A\eta^{h,m}_s(x)+D\bar{F}^m(\bar{X}^m_s)\cdot \eta^{h,m}_s(x)+D_{\xi}\left[\bar{X}^m_s(x)\eta^{h,m}_s(x)\right]\right)\Big] ds\\
&&\!\!\!\!\!\!\!\!\!\!\!\!\!\!\!\!+\frac{1}{2}\sum^m_{k=1}\int^t_0 \EE\Big[D^3\phi(\bar{X}^m_s)\cdot (\eta^{h,m}_s(x), \sqrt{Q_1}e_k, \sqrt{Q_1}e_k) \Big]ds,
\end{eqnarray*}
which implies that
\begin{eqnarray*}
\partial_t(D_x u_m(t,x))\cdot h=\!\!\!\!\!\!\!\!&& \EE\Big[ D^2\phi(\bar{X}^m_t)\cdot \left (\eta^{h,m}_t(x), A\bar{X}^m_t+B^m(\bar{X}^m_t)+\bar{F}^m(\bar{X}^m_t)\right)\Big]\\
&&+\EE\Big[D\phi(\bar{X}^m_t)\cdot\left(A\eta^{h,m}_t(x)+D\bar{F}^m(\bar{X}^m_t)\cdot \eta^{h,m}_t(x)+D_{\xi}\left[\bar{X}^m_t(x)\eta^{h,m}_t(x)\right]\right)\Big]\\
&&+\frac{1}{2}\sum^m_{k=1}\EE\Big[D^3\phi(\bar{X}^m_t)\cdot (\eta^{h,m}_t(x), \sqrt{Q_1}e_k, \sqrt{Q_1}e_k) \Big].
\end{eqnarray*}
Combining \eref{ELemDeriEtah 01}, \eref{Dxx barX}, \eref{A.1.0}, \eref{A.1.2}, \eref{COXT} and \eref{barXm2}, we get for any $t\in (0,T]$ and $\gamma\in (1,\frac{3}{2})$,
\begin{eqnarray*}
|\partial_t(D_x u_m(t,x))\cdot h|=\!\!\!\!\!\!\!\!&& C\EE\left[|\eta^{h,m}_t(x)|(1+\|\bar{X}^m_t\|_2+|\bar{X}^m_t|+\|\bar{X}^m_t\|_{3/2-\gamma}\|\bar{X}^m_t\|_{\gamma})\right]\nonumber\\
&&+C\EE(\|\eta^{h,m}_t(x)\|_2+\|\bar{X}^m_t(x)\|_1\|\eta^{h,m}_t(x)\|_1)\\
\leq\!\!\!\!\!\!\!\!&& C_T t^{-1+\theta/2}\|h\|_{\theta}e^{\delta|x|^2}(1+\|x\|_{\theta}),
\end{eqnarray*}
which proves \eref{UE2}. The proof is complete.
\end{proof}

\vspace{0.2cm}
In order to study the regularity of the solution of Poisson equation \eref{WPE}, we need the following result, whose proof can be founded in \cite[Proposition 3.3]{GSX}.
\begin{lemma} \label{ergodicity in Hfinite}
For any measurable function $\varphi: H^{\beta}\rightarrow H^{\alpha}$ satisfying
$$
\|\varphi(x)-\varphi(y)\|_{\alpha}\leq C(1+\|x\|_{\beta}+\|y\|_{\beta})\|x-y\|_{\beta},
$$
where $\alpha,\beta\in (0,1)$. Then it holds that for any $t>0$,
\begin{eqnarray}
&&\sup_{m\geq 1}\left\| P^{x,m}_t\varphi(y)-\mu^{x,m}(\varphi)\right\|_{\alpha}\nonumber\\
\leq\!\!\!\!\!\!\!\!&& C\left[(t^{-\beta}+t^{-\beta/2})e^{-\frac{\lambda_{1}t}{4}}+(t^{-\frac{\beta}{2}}+1)e^{-\frac{(\lambda_1-L_G)t}{2}}\right](1+|x|^2+|y|^2). \label{ergodicity2}
\end{eqnarray}
\end{lemma}

Then we are in a position to study the regularity of its solution, which plays an important roles in the proof of our main result.
\begin{proposition}\label{P3.6}
Define
\begin{eqnarray}
\Phi^t_m(s,x,y):=\int^{\infty}_{0} D_x \tilde{u}^t_m(s, x)\cdot \left[F^m(x,Y^{x,y,m}_r)-\bar{F}^m(x)\right] dr.\label{SPE}
\end{eqnarray}
Then $\Phi^t_m(s,x,y)$ is a solution of equation \eref{WPE}. Moreover for any $t\in [0,T]$, $h,k\in H_m$ , there exist $C_T>0$ and small enough $\delta>0$ such that
\begin{eqnarray}
&&\sup_{m\geq 1}|\partial_s \Phi^t_m(s,x,y)|\leq C_{T}(t-s)^{-1+\alpha/2}e^{\delta |x|^2}\!(1+\|x\|_{\alpha})(1+|x|^2+|y|^2), s\in [0,t);\label{E120}\\
&&\sup_{s\in [0, t],x\in H_m,m\geq 1}|\Phi^t_m(s,x,y)|\leq C_T(1+|y|)e^{\delta|x|^2};\label{E121}\\
&&\sup_{s\in [0, t],x\in H_m,m\geq 1}|D_x \Phi^t_m(s,x,y)\cdot h|\leq C_{T}(1+|y|)e^{\delta |x|^2}|h|;\label{E122}\\
&&\sup_{s\in [0, t],x\in H_m,m\geq 1}|D_{xx}\Phi^t_m(s,x, y)\cdot(h,k)|\leq \!\!C_{T}(1+|y|)e^{\delta |x|^2}|h|\|k\|_{\tau}, \label{E221}
\end{eqnarray}
where $\tau$ and $\alpha$ are the constants in conditions \ref{A2} and \ref{A3} respectively.
\end{proposition}
\begin{proof}
Note the results in Lemma \ref{Lemma3.6}, we omit the detailed proofs of \eref{E121}-\eref{E221} since it follows almost the same argument in the proof of \cite[Proposition 3.4]{GSX}. Thus we here only give the proofs of \eref{E120}.

In fact, by condition \ref{A3}, \eref{ergodicity2} and \eref{UE2}, it follows that
\begin{eqnarray*}
|\partial_s \Phi^t_m(s,x,y)|\leq\!\!\!\!\!\!\!\!&&\int^{\infty}_{0}\langle \EE F^m(x,Y^{x,y,m}_r)-\bar{F}^m(x), \partial_s(D_x \tilde{u}^t_m(s, x))\rangle dr\\
\leq\!\!\!\!\!\!\!\!&&C_T(t-s)^{-1+\alpha/2}e^{\delta |x|^2}(1+\|x\|_{\alpha})\int^{\infty}_{0}\|\EE F^m(x,Y^{x,y,m}_r)-\bar{F}^m(x)\|_{\alpha}dr\\
\leq\!\!\!\!\!\!\!\!&& C_T(t-s)^{-1+\alpha/2}e^{\delta |x|^2}(1+\|x\|_{\alpha})(1+|x|^2+|y|^2)\\
&&\quad\cdot\int^{\infty}_{0}\left[(r^{-\beta}+r^{-\beta/2})e^{-\frac{\lambda_{1}r}{4}}+(r^{-\frac{\beta}{2}}+1)e^{-\frac{(\lambda_1-L_G) r}{2}}\right]dr\\
\leq\!\!\!\!\!\!\!\!&& C_{T}(t-s)^{-1+\alpha/2}e^{\delta |x|^2}(1+\|x\|_{\alpha})(1+|x|^2+|y|^2).
\end{eqnarray*}
The proof is complete.

\end{proof}

\subsection{The proof of Theorem \ref{main result 2}}

\noindent
\begin{proof} By the discussion in subsection 2.2, we have
\begin{eqnarray*}
\sup_{t\in[0,T]}\left|\EE\phi(X^{m,\vare}_{t})-\EE\phi(\bar{X}^m_{t})\right|
\leq\!\!\!\!\!\!\!\!&&\vare\Bigg[\sup_{t\in [0,T]}|\Phi^t_m(0, x,y)|+\sup_{t\in[0,T]}\left|\EE\Phi^t_m(t, X_{t}^{m,\vare},Y^{m,\vare}_{t})\right|\nonumber\\
&&+\sup_{t\in [0,T]}\EE\int^t_0 \left|\partial_s \Phi^t_m(s, X_{s}^{m,\vare},Y^{m,\vare}_{s})\right|ds\nonumber\\
&&+\sup_{t\in [0,T]}\EE\int^t_0\left|\mathscr{L}^m_{1}(Y^{m,\vare}_{s})\Phi^t_m(s, X_{s}^{m,\vare},Y^{m,\vare}_{s})\right|ds\Bigg]\\
:=\!\!\!\!\!\!\!\!&&\vare\sum^4_{k=1}\Lambda^{m,\vare}_k(T).
\end{eqnarray*}

For the terms $\Lambda_1^{m,\varepsilon}(T)$ and  $\Lambda_2^{m,\varepsilon}(T)$. By estimates \eref{E121},  \eref{AYvare} and \eref{X11}, we have
\begin{eqnarray}
\Lambda^{m,\vare}_1(T)+\Lambda^{m,\vare}_2(T)\leq \!\!\!\!\!\!\!\!&&  C_{T}(1+|y|)e^{\delta|x|^2}+C_{T}\EE\left[(1+|y|+|Y^{m,\vare}_{t}|)e^{\delta|X^{m,\vare}_{t}|^2}\right]\nonumber\\
\leq \!\!\!\!\!\!\!\!&&C_T(1+|y|)e^{\delta|x|^2}.\label{Lambda12(T)}
\end{eqnarray}

For the term $\Lambda_3^{m,\varepsilon}(T)$. By estimates \eref{E120}, \eref{AYvare} and \eref{X11} , we have
\begin{eqnarray}
\Lambda^{m,\vare}_3(T)\leq\!\!\!\!\!\!\!\!&& C_{T}\sup_{t\in[0,T]}\int^t_0(t-s)^{-1+\alpha/2}\EE\left[(1+|X^{m,\vare}_{s}|^2+|Y^{m,\vare}_{s}|^2)(1+\|X^{m,\vare}_{s}\|_{\alpha})e^{\delta |X^{m,\vare}_{s}|^2}\right]ds\nonumber\\
\leq\!\!\!\!\!\!\!\!&&C_T\sup_{s\in [0,T]}\left[\EE(1+|X^{m,\vare}_{s}|^8+|Y^{m,\vare}_{s}|^{8})\right]^{\frac{1}{4}}\sup_{s\in [0,T]}\left[\EE e^{4\delta |X^{m,\vare}_{s}|^2}\right]^{\frac{1}{4}}\nonumber\\
&& \cdot \sup_{t\in[0,T]}\int^t_0(t-s)^{-1+\alpha/2}\left[\EE(1+\|X^{m,\vare}_{s}\|^{2}_{\alpha})\right]^{\frac{1}{2}}ds\nonumber\\
\leq\!\!\!\!\!\!\!\!&&C_T(1+|x|^2+|y|^2)e^{\delta|x|^2}\sup_{t\in[0,T]}\int^t_0(t-s)^{-1+\alpha/2}s^{-\alpha/2}ds\nonumber\\
\leq\!\!\!\!\!\!\!\!&&C_T(1+|x|^2+|y|^2)e^{\delta|x|^2}.\label{LambdA2(T)}
\end{eqnarray}

For the term $\Lambda^{m,\vare}_4(T)$. It is easy to see
\begin{eqnarray}
\Lambda^{m,\vare}_4(T)\leq\!\!\!\!\!\!\!\!&&\sup_{t\in[0,T]}\int^t_0 \EE\left|D_x\Phi^t_m(X_{s}^{m,\vare},Y^{m,\vare}_{s})\cdot \left[AX_{s}^{m,\vare}+B(X_{s}^{m,\vare})\right] \right|ds\nonumber\\
&&+\sup_{t\in[0,T]}\int^t_0 \EE\left|D_x\Phi^t_m(X_{s}^{m,\vare},Y^{m,\vare}_{s})\cdot F^m(X_{s}^{m,\vare},Y^{m,\vare}_{s}) \right|ds\nonumber\\
&&+\sup_{t\in[0,T]}\int^t_0\EE\left|\frac{1}{2}\sum_{k=1}^m \left[D_{xx}\Phi_m^t(X_{s}^{m,\vare},Y^{m,\vare}_{s})\cdot\left(\sqrt{Q_1}e_k,\sqrt{Q_1}e_k\right)\right]\right|ds\nonumber\\
:=\!\!\!\!\!\!\!\!&&\sum^3_{i=1}\Lambda_{4i}^{m,\varepsilon}(T).\label{LambdA2(T)}
\end{eqnarray}
By \eref{E122}, \eref{AYvare}-\eref{Xvare2}, we get
\begin{eqnarray}
\Lambda^{m,\vare}_{41}(T)\leq\!\!\!\!\!\!\!\!&&C_{T}\sup_{t\in[0,T]}\EE\int^t_0 (\|X_{s}^{m,\vare}\|_2+|B(X_{s}^{m,\vare})|)e^{\delta|X_{s}^{m,\vare}|^2}(1+|Y^{m,\vare}_{s}|)ds\nonumber\\
\leq\!\!\!\!\!\!\!\!&&C_{T}\sup_{t\in[0,T]}\int^t_0 \left[\EE(\|X_{s}^{m,\vare}\|^2_2+\|X_{s}^{m,\vare}\|^2_{3/2-\gamma}\|X_{s}^{m,\vare}\|^2_{\gamma})\right]^{1/2}\nonumber\\
&&\quad\quad\quad\quad\quad\cdot \left[\EE(1+|Y^{m,\vare}_{s}|^4)\right]^{1/4}\left[\EE e^{4\delta|X_{s}^{m,\vare}|^2}\right]^{1/4}ds\nonumber\\
\leq\!\!\!\!\!\!\!\!&&C_{T}(1+\|x\|_{\theta}+|y|)e^{\delta|x|^2}\sup_{t\in [0,T]}\int^t_0 s^{-1+\frac{\theta}{2}}+s^{\frac{\alpha-\beta}{2}}ds\nonumber\\
\leq\!\!\!\!\!\!\!\!&&C_{T}(1+\|x\|_{\theta}+|y|)e^{\delta|x|^2}.
\end{eqnarray}
By  \eref{E122} and \eref{E221}, we have
\begin{eqnarray}
\Lambda^{m,\vare}_{42}(T)\leq\!\!\!\!\!\!\!\!&&C_{T}\int^T_0\left[ \EE(1+|X_{s}^{m,\vare}|^2+|Y^{m,\vare}_{s}|^2)\right]^{1/2}\left[\EE e^{2\delta|X_{s}^{m,\vare}|^2}\right]^{1/2}ds\nonumber\\
\leq\!\!\!\!\!\!\!\!&&C_{T}(1+|x|^{2}+|y|^{2})e^{\delta|x|^2}.
\end{eqnarray}
and
\begin{eqnarray}
\Lambda^{m,\vare}_{43}(T)\leq\!\!\!\!\!\!\!\!&&C_T  \sum^{\infty}_{k=1}(\alpha_{1,k}\lambda^{\tau/2}_k)\int^T_0\left[\EE(1+|Y^{m,\vare}_s|^2)\right]^{1/2}\left[\EE e^{2\delta|X_{s}^{m,\vare}|^2}\right]^{1/2}ds\nonumber\\
\leq\!\!\!\!\!\!\!\!&&C_T(1+|x|^2+|y|^2)e^{\delta|x|^2}.
\label{F5.22}
\end{eqnarray}

Finally, combining estimates \eref{Lambda12(T)}-\eref{F5.22}, we get
\begin{eqnarray*}
\sup_{t\in[0,T],m\geq 1}\left|\EE\phi(X^{m,\vare}_{t})-\EE\phi(\bar{X}^m_{t})\right|\leq C \vare,
\end{eqnarray*}
where $C$ is a positive constant only depends on $T$, $x$, $\|x\|_{\theta}$ and $|y|$. The proof is complete.

\end{proof}
\section{Appendix}

In this section, we give some properties of the $b$ and $B$ (see \cite{DX}).
\begin{lemma} \label{Property B0}
For any $x, y \in H^1_0(0, 1)$, it holds that
$$ b(x,x,y)=-\frac12b(x,y,x),\quad b(x,x,x)=0.$$
\end{lemma}

\begin{lemma} \label{Property B1}
Suppose that $\alpha_{i}\geq 0~(i=1,2,3)$
satisfies one of the following conditions \\
$(1) ~\alpha_{i}\neq\frac{1}{2}(i=1,2,3), \alpha_{1}+\alpha_{2}+\alpha_{3}\geq \frac{1}{2}$, \\
$(2) ~\alpha_{i}=\frac{1}{2}$ for some $i$, $\alpha_{1}+\alpha_{2}+\alpha_{3}>\frac{1}{2}$,\\
then $b$ is continuous from $H^{\alpha_{1}}(0,1)\times H^{\alpha_{2}+1}(0,1)\times H^{\alpha_{3}}(0,1)$ to $\mathbb{R}$, i.e.
$$\big|b(x,y,z)\big|\leq C\|x\|_{\alpha_{1}}\|y\|_{\alpha_{2}+1}\|z\|_{\alpha_{3}}.$$
\end{lemma}

By Lemmas \ref{Property B0} and \ref{Property B1}, it is easy check that
\begin{lemma} \label{Property B3}
For any $x, y \in H^1_0(0, 1)$, it holds that for any $\gamma\in(1/2,1)$,
\begin{eqnarray}
&&|B(x)|\leq C\|x\|_{1-\gamma}\|x\|_{1/2+\gamma},\label{BP1}\\
&&\|B(x)-B(y)\|_{-1}\leq C(\|x\|_{\gamma}+\|y\|_{\gamma})|x-y|.\label{BP2}
\end{eqnarray}
\end{lemma}
At the end of this section, we give the classical Gronwall inequality and a Gronwall-type inequality.

\begin{lemma}[Gronwall inequality] \label{Gronwall 1}
Let $\alpha, \beta$ be real-value functions defined on $[0, T]$, assume that $\beta$ and $u$ are continuous and that the negative part of $\alpha$ is integrable on every closed and bounded subinterval of $[0, T]$. \\
(a) If $\beta$ is non-negative and if $u$ satisfies the integral inequality
\begin{eqnarray*}
u_t\leq \alpha_t+\int^t_0\beta_s u_sds, \quad \forall t\in [0, T],
\end{eqnarray*}
then
$$
u_t\leq \alpha_t+\int^t_0 \alpha_s \beta_s\exp{\left(\int^t_s \beta_rdr\right)}ds,\quad \forall t\in [0, T].
$$
(b) If, in addition, the function $\alpha$ is nondecreasing, then
$$
u_t\leq \alpha_t\exp{\left(\int^t_0 \beta_rdr\right)},\quad \forall t\in [0, T].
$$
\end{lemma}

\begin{lemma}[Gronwall-type inequality] \label{Gronwall 2}
Let $f(t)$ be a non-negative real-valued integrable function on $[0,T]$.
For any given $\alpha, \beta \in[0,1)$,
if there exist two positive constants $C_1, C_2$ such that
\begin{align}
f(t)\leq C_1t^{-\alpha}+C_2\int_{0}^{t}(t-s)^{-\beta}f(s)ds, \quad \forall t\in[0,T], \nonumber
\end{align}
then there exists some $k\in \mathbb{N}$ and a positive constant $C:=C_{\alpha, \beta, T}$
such that
\begin{align}
f(t)\leq C C_1t^{-\alpha}e^{C C^{k}_2}, \quad \forall t\in[0,T].  \nonumber
\end{align}
\end{lemma}

\vspace{0.3cm}
For the Galerkin approximation \eref{Ga mainE} of system \eref{main equation}, we have the following approximation.

\begin{lemma} \label{GA1}
For any $\vare>0, (x,y)\in H\times H$, we have
\begin{eqnarray}
\lim_{m\rightarrow \infty}\EE\left(\sup_{t\in[0,T]}|X^{m,\vare}_t-X^{\vare}_t|^p\right)=0,\label{FA1}
\end{eqnarray}
\begin{align}
\lim_{m\rightarrow \infty}\EE\left(\sup_{t\in[0,T]}|\bar{X}^{m}_t-\bar{X}_t|^p\right)=0. \label{FA2}
\end{align}
\end{lemma}

\begin{proof}
Since the proofs of \eref{FA1} and \eref{FA2} almost the same,  we only prove \eref{FA1} here. Note that for any $t\geq 0$,
\begin{eqnarray*}
Y^{m,\vare}_t-Y^{\vare}_t=\!\!\!\!\!\!\!\!&&e^{tA/\vare}(y^m-y)+\frac{1}{\vare}\int^{t}_{0}e^{(t-s)A/\vare}(\pi_m-I)G(X^{\vare}_s, Y^{\vare}_s)ds \nonumber\\ &&+\frac{1}{\vare}\int^{t}_{0}e^{(t-s)A/\vare}[G^{m}(X^{m,\vare}_s,Y^{m,\vare}_s)-G^{m}(X^{\vare}_s,Y^{\vare}_s)]ds\\
&&+\frac{1}{\sqrt{\varepsilon}}\int^t_0 e^{(t-s)A/\vare}\sqrt{Q_2}d \bar{W}^{2,m}_s-\frac{1}{\sqrt{\varepsilon}}\int^t_0 e^{(t-s)A/\vare}\sqrt{Q_2}dW^2_s.
\end{eqnarray*}
It is clear that for any $p\geq 1$ and $\vare>0$,
\begin{eqnarray*}
|Y^{m,\vare}_t-Y^{\vare}_t|^p\leq\!\!\!\!\!\!\!\!&&C_p|y^m-y|^p+C_{p,\vare}\int^{t}_{0}|(\pi_m-I)G(X^{\vare}_s, Y^{\vare}_s)|^pds\nonumber\\
&&+C_{p,\vare}\int^t_0|G^{m}(X^{m,\vare}_s,Y^{m,\vare}_s)-G^{m}(X^{\vare}_s,Y^{\vare}_s)|^pds\nonumber\\
&&+C_{p,\vare}\left|\int^t_0 e^{(t-s)A/\vare}\sqrt{Q_2}d \bar{W}^{2,m}_s-\int^t_0 e^{(t-s)A/\vare}\sqrt{Q_2}dW^2_s\right|^p\nonumber\\
\leq\!\!\!\!\!\!\!\!&&C_p|y^m-y|^p+C_{p,\vare}\int^{t}_{0}|(\pi_m-I)G(X^{\vare}_s, Y^{\vare}_s)|^pds\nonumber\\
&&+C_{p,\vare}\int^t_0|X^{m,\vare}_s-X^{\vare}_s|^pds+C_{p,\vare}\int^t_0|Y^{m,\vare}_s-Y^{\vare}_s|^pds\nonumber\\
&&+C_{p,\vare}\left|\int^t_0 e^{(t-s)A/\vare}\sqrt{Q_2}d \bar{W}^{2,m}_s-\int^t_0 e^{(t-s)A/\vare}\sqrt{Q_2}dW^2_s\right|^p.
\end{eqnarray*}
Gronwall's inequality implies that for any $t\in [0,T]$,
\begin{eqnarray}
|Y^{m,\vare}_t-Y^{\vare}_t|^p\leq\!\!\!\!\!\!\!\!&&C_{p,T,\vare}|y^m-y|^p+C_{p,T,\vare}\int^{T}_{0}|(\pi_m-I)G(X^{\vare}_s, Y^{\vare}_s)|^pds\nonumber\\
&&+C_{p,T,\vare}\int^t_0\EE|X^{m,\vare}_s-X^{\vare}_s|^pds\nonumber\\
&&+C_{p,T,\vare}\sup_{t\in [0,T]}\left|\int^t_0 \!\!e^{(t-s)A/\vare}\sqrt{Q_2}d \bar{W}^{2,m}_s-\int^t_0 \!\!e^{(t-s)A/\vare}\sqrt{Q_2}dW^2_s\right|^p.\label{Y^m-Y}
\end{eqnarray}

By a simple computation, we get
\begin{eqnarray*}
X^{m,\vare}_t-X^{\vare}_t=\!\!\!\!\!\!\!\!&&e^{tA}(x^m-x)+\int^{t}_{0}e^{(t-s)A}(\pi_m-I)B(X^{\vare}_s)ds\nonumber\\
&&+\int^{t}_{0}e^{(t-s)A}\left[B^{m}(X^{m,\vare}_s)-B^{m}(X^{\vare}_s)\right]ds\nonumber\\
&&+\int^{t}_{0}e^{(t-s)A}(\pi_m-I)F(X^{\vare}_s, Y^{\vare}_s)ds\nonumber\\
&&+\int^{t}_{0}e^{(t-s)A}\left[F^{m}(X^{m,\vare}_s,Y^{m,\vare}_s)-F^{m}(X^{\vare}_s,Y^{\vare}_s)\right]ds\nonumber\\
&&+\int^t_0 e^{(t-s)A}\sqrt{Q_1}d \bar W^{1,m}_s-\int^t_0 e^{(t-s)A}\sqrt{Q_1}d  W^{1}_s.
\end{eqnarray*}

We now define the following stopping time
$$
\tau_k=\inf\{t\geq 0: |X^{m,\vare}_t|+|X^{\vare}_t|\geq k\}.
$$
Note that by Sobolev's embedding theorem,
\begin{eqnarray*}
\left|e^{(t-s)A}\left[B^{m}(X^{m,\vare}_s)-B^{m}(X^{\vare}_s)\right]\right|\leq\!\!\!\!\!\!\!\!&&C\left|e^{(t-s)A}\left[D_{\xi}\left((X^{m,\vare}_s)^2-(X^{\vare}_s)^2)\right)\right]\right|\\
\leq\!\!\!\!\!\!\!\!&&C\left\|e^{(t-s)A}\left[D_{\xi}\left((X^{m,\vare}_s)^2-(X^{\vare}_s)^2)\right)\right]\right\|_{H^{1/2, 1}}\\
\leq\!\!\!\!\!\!\!\!&&C\int^{t}_{0} (t-s)^{-\frac{3}{4}}\|(X^{m,\vare}_s+X^{\vare}_s)(X^{m,\vare}_s-X^{\vare}_s)\|_{L^1}ds,\\
\leq\!\!\!\!\!\!\!\!&&C\int^{t}_{0} (t-s)^{-\frac{3}{4}}(|X^{m,\vare}_s|+|X^{\vare}_s|)|X^{m,\vare}_s-X^{\vare}_s|ds.
\end{eqnarray*}
Then we have for any $p\geq 4$ and $\frac{1}{p}+\frac{1}{q}=1$,
\begin{eqnarray*}
&&\sup_{t\in [0,T\wedge \tau_k]}\left|\int^{t}_{0}e^{(t-s)A}\left[B^{m}(X^{m,\vare}_s)-B^{m}(X^{\vare}_s)\right]ds\right|\\
\leq\!\!\!\!\!\!\!\!&&C_k\sup_{t\leq T\wedge \tau_k}\left|\int^{t}_{0} (t-s)^{-\frac{3}{4}}|X^{m,\vare}_s-X^{\vare}_s|ds\right|\nonumber\\
\leq\!\!\!\!\!\!\!\!&&C_k\sup_{t\in [0,T\wedge \tau_k]}\left|\left[\int^{t}_{0} (t-s)^{-\frac{3 q}{4}}ds\right]^{1/q}\left[\int^t_0|X^{m,\vare}_s-X^{\vare}_s|^p ds\right]^{1/p}\right|\nonumber\\
\leq\!\!\!\!\!\!\!\!&&C_{k,T}\sup_{t\in [0,T\wedge \tau_k]}\left[\int^t_0|X^{m,\vare}_s-X^{\vare}_s|^p ds\right]^{1/p},
\end{eqnarray*}
which implies that
\begin{eqnarray}
\sup_{t\in [0,T\wedge \tau_k]}\left|\int^{t}_{0}e^{(t-s)A}\left[B^{m}(X^{m,\vare}_s)-B^{m}(X^{\vare}_s)\right]ds\right|^p
\leq\!\!\!\!\!\!\!\!&&C_{k,T} \int^{T\wedge \tau_k}_0|X^{m,\vare}_s-X^{\vare}_s|^p ds.
\end{eqnarray}
Hence by \eref{Y^m-Y}, we get that
\begin{eqnarray}
&&\sup_{t\in [0,T\wedge \tau_k]}|X^{m,\vare}_t-X^{\vare}_t|^p\nonumber\\
\leq\!\!\!\!\!\!\!\!&&C_p|x^m-x|^p+C_{p,T,\vare}|y^m-y|^p\nonumber\\
&&+C_{k,p,T,\vare}\int^{T\wedge \tau_k}_{0}|X^{m,\vare}_s-X^{\vare}_s|^p ds+\sup_{t\in [0,T\wedge \tau_k]}\left|\int^{t}_{0}e^{(t-s)A}(\pi_m-I)B(X^{\vare}_s)ds\right|^p\nonumber\\
&&+C_{p,T}\!\int^{T\wedge \tau_k}_{0}|(\pi_m-I)F(X^{\vare}_s, Y^{\vare}_s)|^pds\!+\!C_{p,T,\vare}\!\int^{T\wedge \tau_k}_{0}\!\!\EE|(\pi_m-I)G(X^{\vare}_s, Y^{\vare}_s)|^p ds\nonumber\\
&&+C_p\sup_{t\in [0,T]}\left|\int^t_0 e^{(t-s)A}\sqrt{Q_1}d \bar W^{1,m}_s-\int^t_0 e^{(t-s)A}\sqrt{Q_1}dW^1_s\right|^p\nonumber\\
&&+C_{p,T,\vare}\sup_{t\in[0,T]}\left|\int^t_0 e^{(t-s)A/\vare}\sqrt{Q_2}d \bar{W}^{2,m}_s-\int^t_0 e^{(t-s)A/\vare}\sqrt{Q_2}dW^2_s\right|^p.\label{X^m-X}
\end{eqnarray}
Then by Gronwall's inequality and taking expectation, we get
\begin{eqnarray}
&&\EE\left(\sup_{t\in [0,T\wedge \tau_k]}|X^{m,\vare}_t-X^{\vare}_t|^p\right)\nonumber\\
\leq\!\!\!\!\!\!\!\!&&C_{p,k,T,\vare}|x^m-x|^p+C_{p,T}|y^m-y|^p+C_{p,k,T,\vare}\EE\left|\int^{T}_{0}(\pi_m-I)|B(X^{\vare}_s)|ds\right|^p\nonumber\\
&&+C_{p,k,T,\vare}\!\int^{T}_{0}\!\!\EE|(\pi_m-I)F(X^{\vare}_s, Y^{\vare}_s)|^pds\!+\!C_{p,T}\!\int^{T}_{0}\!\!\EE|(\pi_m-I)G(X^{\vare}_s, Y^{\vare}_s)|^p ds\nonumber\\
&&+C_{p,k,T,\vare}\EE\left[\sup_{t\in [0,T]}\left|\int^t_0 e^{(t-s)A}\sqrt{Q_1}d \bar W^{1,m}_s-\int^t_0 e^{(t-s)A}\sqrt{Q_1}dW^1_s\right|^p\right]\nonumber\\
&&+C_{p,k,T,\vare}\EE\left[\sup_{t\in[0,T]}\left|\int^t_0 e^{(t-s)A/\vare}\sqrt{Q_2}d \bar{W}^{2,m}_s-\int^t_0 e^{(t-s)A/\vare}\sqrt{Q_2}dW^2_s\right|^p\right].\label{X^m-X}
\end{eqnarray}

By the a priori estimate of $(X^{\vare}_s, Y^{\vare}_s)$ in \eref{L6.1} and the dominated convergence theorem, we have
$$\lim_{m\rightarrow \infty}\EE\left|\int^{T}_{0}(\pi_m-I)|B(X^{\vare}_s)|ds\right|^p=0,$$
$$\lim_{m\rightarrow \infty}\int^{T}_{0}\EE|(\pi_m-I)F(X^{\vare}_s, Y^{\vare}_s)|^pds=0,$$
$$\lim_{m\rightarrow \infty}\int^{T}_{0}\EE|(\pi_m-I)G(X^{\vare}_s, Y^{\vare}_s)|^p ds=0,$$
$$\lim_{m\rightarrow \infty}\EE\left[\sup_{t\in [0,T]}\left|\int^t_0 e^{(t-s)A}\sqrt{Q_1}d \bar W^{1,m}_s-\int^t_0 e^{(t-s)A}\sqrt{Q_1}dW^1_s\right|^p\right]=0,$$
$$\lim_{m\rightarrow \infty}\EE\left[\sup_{t\in [0,T]}\left|\int^t_0 e^{(t-s)A/\vare}\sqrt{Q_2}d \bar{W}^{2,m}_s-\int^t_0 e^{(t-s)A/\vare}\sqrt{Q_2}dW^2_s\right|^p\right]=0.$$
Summarizing the above, we get that for $\vare>0$,
\begin{eqnarray*}
\lim_{m\rightarrow \infty}\EE\left(\sup_{t\in [0,T\wedge \tau_k]}|X^{m,\vare}_t-X^{\vare}_t|^p\right)=0.
\end{eqnarray*}
Finally, by Chebyshev's inequality we get
\begin{eqnarray*}
&&\EE\left(\sup_{t\in [0,T]}|X^{m,\vare}_t-X^{\vare}_t|^p\right)\\
\leq\!\!\!\!\!\!\!\!&&\EE\left(\sup_{t\in [0,T\wedge \tau_k]}|X^{m,\vare}_t-X^{\vare}_t|^p\right)+\left[\PP(T>\tau_k)\right]^{1/2}\left[\EE\left(\sup_{t\in [0,T]}|X^{m,\vare}_t-X^{\vare}_t|^{2p}\right)\right]^{1/2}\nonumber\\
\leq\!\!\!\!\!\!\!\!&&\EE\left(\sup_{t\in [0,T\wedge \tau_k]}|X^{m,\vare}_t-X^{\vare}_t|^p\right)+\left[\frac{\EE\left(\sup_{t\in [0,T]}\left(|X^{m,\vare}_t|+|X^{\vare}_t|\right)\right)}{k}\right]^{1/2}\\
&&\cdot\left[\EE\left(\sup_{t\in [0,T]}|X^{m,\vare}_t-X^{\vare}_t|^{2p}\right)\right]^{1/2}.
\end{eqnarray*}
Letting $m\rightarrow \infty $ firstly, then $k\rightarrow \infty$. We final obtain that \eref{FA1} holds. The proof is complete.
\end{proof}

\vspace{0.3cm}

\textbf{Acknowledgment}.
Peng Gao is supported by the Fundamental Research Funds for the Central Universities (2412020FZ022).
Xiaobin Sun is supported by the National Natural Science Foundation of China (11931004, 12090011), the QingLan Project and the Priority Academic Program Development of Jiangsu Higher Education Institutions.

\end{document}